\documentclass[12pt,a4paper,final]{article}
\usepackage{amsmath,amsfonts,latexsym,amssymb,amsthm}
\usepackage{graphicx}
\usepackage{geometry}                % See geometry.pdf to learn the layout options. There are lots.
\usepackage[utf8x]{inputenc}

\newtheorem{thm}{Theorem}[section]
\newtheorem{defi}{Definition}[section]

\newtheorem{hyp}{Assumption}[section]
\newtheorem{cor}{Corollary}[section]
\newtheorem{lem}{Lemma}[section]
\newenvironment{remark}[1][Remark]{\begin{trivlist}
\item[\hskip \labelsep {\bfseries #1}]}{\end{trivlist}}

\newcommand{\ind}{1\!\!1}

\title{Parametric estimation for Gaussian fields indexed by graphs}

\author{T. Espinasse, F. Gamboa and  J-M. Loubes}
\begin{document}
\maketitle \noindent
\begin{abstract}
In this paper, using spectral theory of Hilbertian operators, we study $ARMA$ Gaussian processes indexed by graphs. We extend Whittle maximum likelihood estimation of the parameters for the corresponding \textit{spectral density} and show their asymptotic optimality.
\end{abstract}

% \end{frontmatter}

\tableofcontents

% FIXME : dessin (regular pattern), ordre, introduction

\newpage

\section*{\\Introduction}
In the past few years, much interest has been paid to the study of random fields over graphs. It has been driven by the growing needs for both theoretical and practical results for data indexed by graphs. On the one hand, the definition of graphical models by J.N. Darroch, S.L. Lauritzen and T.P. Speed in $1980$ \cite{graph_mod} fostered new interest in Markov fields, and many tools have been developed in this direction (see, for instance~\cite{verzelen1} and~\cite{verzelen2}). On the another hand, the industrial demand linked to graphical problems has risen with the apparition of new technologies. In very particular, the Internet and social networks provide a huge field of applications, but  biology, economy, geography or image analysis also benefit from models taking into account a graph structure.\\
\indent The analysis of road traffic is at the root of this work. Actually, prediction of road traffic deals with the forecast of speed of vehicles which may be seen as a spatial random field over the traffic network. Some work has been done without taking into account the particular graph structure of the speed process (see for example \cite{MR2369028} and \cite{MR2328555} for related statistical issues). In this paper, we build a new model for Gaussian random fields over graphs and study statistical properties of such stochastic processes.\vskip .1in
A random field over a graph is a spatial process indexed by the vertices of a graph, namely $(X_i)_{i \in G}$, where $G$ is a given graph. Many models already exist in the probabilistic literature, ranging from Markov fields to autoregressive processes, which are based on two general kinds of construction. On the one hand, graphical models are defined as Markov fields (see for instance~\cite{guyonbook}), with a particular dependency structure. Actually, they are built by specifying a dependency structure for  $X_i$ and $X_j$, conditionally 
to the other variables, as soon as the locations $i \in G$ and $j \in G$ are connected. For graphical models, we refer for instance to~\cite{graph_mod} and references therein. On the other hand, the graph itself, through the adjacency operator, can provide the dependency. This is the case, for example, of autoregressive models on $\mathbb{Z}^d$ (see~\cite{guyonbook}). Here, the local form of the graph is strongly used for statistical inference. \vskip .1in
More precisely, the usual purpose of graphical models is to design an underlying graph which reflects the dependency of the data. This method has to be applied when this graph is not easily known (for instance social networks) or when it plays the role of a model which helps understanding the correlations between high complex data (for instance for biological purpose). Our approach differs since, in our case, the graph is known, and we aim at using a model with \textit{stationary} properties. Indeed, in the case of road traffic, we can consider that the correlations of the process depend mainly on the local structure of the network. This assumption is commonly accepted among professionals of road trafficking speaking of capacity of the road.
% Thus, $ARMA$ models will be better in this framework. 

\vskip .1in
In this paper, 
% we will provide a new framework that takes into account advantages of both points of views. Hence,
 we  extend some classical results from time series to spatial fields over general graphs and provide a new definition for regular $ARMA$ processes on graphs. 
For this, we will make use of spectral analysis and extend to our framework some classical results of time series.
In particular, the notion of spectral density may be extended to graphs. This will enable us to construct a maximum likelihood estimate for parametric models of spectral densities. This also leads to an extension of the Whittle's approximation (see \cite{sego}, \cite{AzDa}). Actually, many extensions of this approximation have been performed, even in non-stationary cases (see \cite{dahlhaus2}, \cite{robinson1}, \cite{robinson2}). The extension studied here concerns general $ARMA$ processes over graphs. We point out that we will compare throughout all the paper our new framework with the case $G= \mathbb{Z}^d, d \geq 1$.\vskip .1in
% The paper falls into the following parts. 
Section \ref{s:rappels} is devoted to some definitions of graphs and spectral theory for time series. Then we state the definition of general $ARMA$ processes over a graph in Section \ref{s:analytic_construction}. The convergence of the Whittle maximum likelihood estimate and its asymptotic efficiency are given in Theorems \ref{thm_conv} and \ref{t:as_norm} in Section \ref{s:main_thm}. 
% Section \ref{s:algebraical_construction}
% % % provides some way to modify the graph itself, in order to consider a larger class of processes, which leads to  the construction of a large class of Markov fields. 
% Section \ref{s:spectral} gives the definition of the spectral measure for the graph, and discuss sufficient conditions for its existence.
Section \ref{s:discussion} is devoted to a short discussion on potential applications and perspectives.
Some simulations are provided in Section \ref{s:simul}. The last section provides all necessary tools to prove the main theorems, in particular Szegö's Lemmas for graphs are given in Section \ref{s:szego}, while the proofs of the technical Lemmas are postponed in Section~\ref{s:technical_lemmas}.

\section{Definitions and useful properties for spectral analysis and Toeplitz operators}\label{s:rappels}

%Gaussian field, 
%Graphical models
%Gibbs fields
\subsection{Graphs, adjacency operator, and spectral representation}

In the whole paper, we will consider a Gaussian spatial process $(X_i)_{i \in G}$ indexed by the vertices of an infinite undirected weighted graph.

We will call $\mathbf{G}=(G,W)$ this graph, where
\begin{itemize}
 \item $G$ is the set of vertices. $\mathbf{G}$ is said to be infinite as soon as $G$ is infinite (but countable). 
\item $W \in [-1,1]^{G\times G}$ is the symmetric weighted adjacency operator.
 That is, $ |W_{ij}|\neq 0$ when $i\in G$ and $j \in G$
are connected.   
\end{itemize}

We assume that $W$ is symmetric ($W_{ij}=W_{ji},\; i,j\in G$) since we deal only with undirected graphs.

For any vertex $i \in G$, a vertex $j \in G$ is said to be a neighbor of $i$ if, and only if, $W_{ij} \neq 0$. The degree $\operatorname{deg}(i)$ of $i$ is the number of neighbors of the vertex $i$,
and the degree of the graph $\mathbf{G}$ is defined as the maximum degree of the vertices of the graph $\mathbf{G}$ :
$$\operatorname{deg}(\mathbf{G}) := \max_{i \in G} \operatorname{deg}(i).$$ 
From now on, we assume that the degree of the graph $\mathbf{G}$ is bounded :
$$\operatorname{deg}(\mathbf{G})< + \infty.$$

Assume now that $W$ is renormalized : its entries belong to $[-\frac{1}{\operatorname{deg}(\mathbf{G})},\frac{1}{\operatorname{deg}(\mathbf{G})}]$. This is not restrictive since re-normalizing the adjacency operator does not change the 
objects introduced later.
In particular, the spectral representation of Hilbertian operator is not sensitive to a renormalization.

Notice that in the classical case $G=\mathbb{Z}$, the renormalized 
adjacency operator is
\begin{equation}
 \label{e:defz}
W^{(\mathbb{Z})}_{ij}=\frac{1}{2}\ind_{\{|i-j|=1\}}, (i,j\in\mathbb{Z}).
\end{equation}

Here, $\operatorname{deg}(\mathbb{Z}) = 2$.
This case will be used in all the paper as an illustration example.

To introduce the spectral decomposition, consider the action of the adjacency operator on $l^2(G)$ as 
$$\forall u \in l^2(G), (Wu)_i := \sum_{j \in G} W_{ij} u_j, (i\in G).$$

We denote by $B_G$ the set of all bounded Hilbertian
 operators on $l^2(G)$ (the set of square sommable real sequences indexed by $G$). 
The operator space $B_G$ will be endowed with the classical operator norm 
$$\forall A \in B_G, \left\| A \right\|_{2,op}: = \sup_{u \in l^2(G), \left\| u\right\|_2 \leq 1} \left\| Au \right\|_2 , $$
where $\left\| . \right\|_2$ stands for the usual norm on $l^2(G)$.

Notice that, as the degree of $\mathbf{G}$ and the entries of $W$ are both bounded, $W$ lies in $B_{G}$, and we have $$\left\| W \right\|_{2,op} \leq 1 .$$  
% $B_G$ is endowed with the classical operator norm 
% $$\forall A \in B_G, \left\| A \right\|_{2,op}: = \sup_{u \in l^2(G), \left\| u\right\|_2} \left\| Au \right\|_2, $$
% where $\left\| . \right\|_2$ stands for the usual norm on $l^2(G)$.

Recall that for any bounded Hilbertian operator $A \in B_G$, the spectrum $\operatorname{Sp}(A)$ is defined as the set of all complex numbers $\lambda$ such that 
$\lambda \operatorname{Id}- A$ is not invertible (here $\operatorname{Id}$ stands for the identity on $l^2(G)$). Since $W$ is bounded and symmetric,
 $\operatorname{Sp}(W)$ is a non-empty compact subset of $\mathbb{R}$ \cite{rudin}.

We aim now at providing a spectral representation of any bounded normal Hilbertian operator. For this, first recall the definition of a
 resolution of identity (see for example \cite{rudin}):

\begin{defi}
 Let $\mathcal{M}$ be a $\sigma$-algebra over a set $\Omega$.
%, and $H$ an Hilbert Space.
 We call identity resolution (on $\mathcal{M}$) a map
$$E : \mathcal{M} \rightarrow B_G $$
such that, 
\begin{enumerate}
 \item $E(\emptyset{}) = 0, E(\Omega)= I$.
\item For any $\omega \in \mathcal{M}$, the operator $E(\omega)$ is a projection operator.
\item For any $\omega,\omega' \in \mathcal{M}$, we have
 $$E(\omega \cap \omega') =E(\omega)E(\omega')=E(\omega')E(\omega).$$
\item For any $\omega,\omega' \in \mathcal{M}$ such that $\omega \cap \omega' = \emptyset$, we have 
$$E(\omega \cup \omega') = E(\omega)+E(\omega').$$
\end{enumerate}
% Therefore, for all $x \in l^2(G)$ and $y \in l^2(G)$, the functional $E_{x,y}$ defined by
% $$E_{xy}(\omega) = \langle E(\omega)x,y\rangle_{l^2(G)}$$
% is a complex measure on $\mathcal{M}$.
%   \end{enumerate}
\end{defi}

We can now recall the fundamental decomposition theorem (see for example \cite{rudin})
\begin{thm}[Spectral decomposition]
\label{thm_spectral_decomposition}
If $A \in B_G$ is symmetric, then there exists a unique identity resolution $E$ over all Borelian subsets of $\operatorname{Sp}(A)$, such that
$$A = \int_{\operatorname{Sp}(A)} \lambda \mathrm{d} E (\lambda) .$$
% Moreover, for any  $U \in B_S$ such that $UA = AU$, every projector $E(\omega), \omega \subset \operatorname{Sp}(A)$ commutes with $U$
\end{thm}

% Since $W$ is self-adjoint in our case, it is a normal operator, so Theorem \ref{thm_spectral_decomposition} may be applied. 

From the last theorem, we obtain the spectral representation of the adjacency operator $W$ thanks to an identity resolution $E$ over the Borelians of $\operatorname{Sp}(W)$ 
$$W = \int_{\operatorname{Sp}(W)} \lambda \mathrm{d} E (\lambda).$$

% Moreover, with this decomposition, we can give a spectral representation for the powers $W^k, k \in \mathbb{Z}$ of $W$
Obviously, we have
$$W^k = \int_{\operatorname{Sp}(W)} \lambda^k \mathrm{d} E (\lambda), k \in \mathbb{N}.$$

 Define now, for any $i \in G$, the sequences $\delta_i$ in $l^2(G)$ by
$$\delta_i := (\ind_{k = i})_{k \in G}.$$

For any $i,j \in G$, the sequences $\delta_i$ and $\delta_j$ define the real measure $\mu_{ij}$ by
$$\forall \omega \subset \operatorname{Sp}(W), \mu_{ij}(\omega) : = 
% E_{\delta_i \delta_j}(\omega) = 
\langle E(\omega)\delta_i,\delta_j\rangle_{l^2(G)}. $$

Hence, we can write : 
$$\forall k \in \mathbb{N}, \forall i,j \in G, \left(W^k\right)_{ij} = \int_{\operatorname{Sp}(W)} \lambda^k \mathrm{d}\mu_{ij}.$$

This family of measures $\mu_{ij},i,j \in G$ will be used in the whole paper. They convey both spectral information of the adjacency operator, and combinatorial information on the number of path and loops in $\mathbf{G}$. Indeed, the quantity $\left(W^k\right)_{ij}$ is the number of path (counted with their weights) going from $i$ to $j$ with length $k$. 

Note also that all diagonals measures $\mu_{ii}, i \in G$ are probability measures.

\subsection{The adjacency operator of $\mathbb{Z}$ and its spectral decomposition}

In the usual case of $\mathbb{Z}$, an explicit expression for $\mu_{ij}$ can be given.

% Recall that $W^{(\mathbb{Z})}$ denotes the renormalized adjacency operator of $\mathbb{Z}$ :
% 
% $$W^{(\mathbb{Z})}_{ij}=\frac{1}{2}\ind_{\{|i-j|=1\}}, (i,j\in\mathbb{Z}).$$ 

Denote $T_{k}(X)$ the $k^\text{th}$-Chebychev polynomial ($k \in \mathbb{N}$). We can provide the spectral decomposition of $W^{(\mathbb{Z})}$ ($W^{(\mathbb{Z})}$ has been defined in Equation \ref{e:defz}).
$$\forall i,j \in \mathbb{Z}, \left(\left(W^{(\mathbb{Z})}\right)^k\right)_{ij} = \frac{1}{\pi} 
\int_{[-1,1]} \lambda^k \frac{T_{\left|j-i\right|}(\lambda)}{\sqrt{1- \lambda^2}} \mathrm{d}\lambda  .$$
This shows that, in this case, and for any $i,j \in G$, the measure $\mathrm{d}\mu_{ij}$ is absolutely continuous with respect to the Lebesgue measure,
 and its density is given by
$$\frac{\mathrm{d}\mu_{ij}}{\mathrm{d}\lambda} = \frac{1}{\pi}\frac{T_{\left|j-i\right|}(\lambda)}{\sqrt{1- \lambda^2}}.$$

Notice that we recover the usual spectral decomposition pushing forward $\mu_{ij}$ by the function $\cos$ :
$$\forall i,j \in G, \mathrm{d}\hat{\mu}_{ij}(t) := \frac{1}{2\pi} \cos\left((j-i)t \right) \mathrm{d}t.  $$
We get 
$$\forall i,j \in \mathbb{Z}, \left(\left(W^{(\mathbb{Z})}\right)^k\right)_{ij} = \int_{[0,2 \pi]} \cos(t)^k  \mathrm{d}\hat{\mu}_{ij}(t) . $$

% This corresponds to another choice for the identity resolution.
% Further on, this enables us to handle the usual case of processes indexed by $\mathbb{Z}$. 

\subsection{Time series, spectral representation, and $MA_\infty$}\label{ss:maforz}

Our aim is to study some kind of stationary processes indexed by the vertices $G$ of the graph $\mathbf{G}$. To begin with, let us recall  the usual case of $\mathbb{Z}$. In particular, let us introduce Toeplitz operators associated to stationary time series.

Let $\mathbf{X} = (X_i)_{i \in\mathbb{Z}}$ be a stationary Gaussian process indexed by $\mathbb{Z}$. Since $\mathbf{X}$ is Gaussian, stationarity is equivalent to second order stationarity, 
that is,
$\forall i,k \in \mathbb{Z},  \operatorname{Cov}(X_i, X_{i+k}) $ does not depend on $i$.
Thus, we can define 
$$r_k :=  \operatorname{Cov}(X_i, X_{i+k}).$$
Aassume further that $(r_k)_{k \in \mathbb{Z}} \in l^1(\mathbb{Z})$.
This leads to a particular form of the covariance operator $\Gamma$ defined on $l^2(\mathbb{Z})$ by 
$$\forall i,j \in \mathbb{Z}, \Gamma_{ij} := r_{i-j}.$$ 
Recall that $B_{\mathbb{Z}}$ denotes here the set of bounded Hilbertian operators on $l^2(\mathbb{Z})$.
Notice that, since $(r_k)_{k \in \mathbb{Z}} \in l^1(\mathbb{Z})$, we have $\Gamma \in B_\mathbb{Z}$ (see for instance \cite{davis} for more details). This bounded operator
is constant over each diagonals, and is therefore called a Toeplitz operator (see also \cite{bottcher} for a general introduction to Toeplitz operators).

% Toeplitz operators enjoy the following representation,
As $(r_k)_{k \in \mathbb{Z}} \in l^1(\mathbb{Z})$, we have
$$\forall i,j \in \mathbb{Z}, \mathcal{T}(g)_{ij}:=\Gamma_{ij} = \frac{1}{2\pi}\int_{[0,2 \pi]} g(t) \cos\left((i-j)t\right)\mathrm{d}t, $$
where $g$ is the spectral density of the process $\mathbf{X}$, defined by
$$g(t) := 2\sum_{k \in \mathbb{N}^*} r_k \cos(kt)+r_0 .$$

This expression can be written, using the Chebychev polynomials $(T_k)_{k \in \mathbb{N}}$, 
$$g(t) := 2\sum_{k \in \mathbb{N}^*} r_k T_k\left(\cos(t)\right)+r_0 T_0\left(\cos(t)\right).$$

% This falls in the framework of Theorem \ref{thm_conv}, by setting
Let, for $\lambda \in [-1,1]$,
\begin{equation}
\label{e:densite}
 f(\lambda) :=  2\sum_{k \in \mathbb{N}^*} r_k T_k(\lambda)+r_0 T_0(\lambda).
\end{equation}
 
We get, using the family $(\hat{\mu}_{ij} )_{i,j \in \mathbb{Z}}$ defined above,
$$\forall i,j \in \mathbb{Z}, \Gamma_{ij} = \int_{[0,2 \pi]} f\left(\cos(t)\right) \mathrm{d}\hat{\mu}_{ij}(t). $$

Notice that the last expression may also be written as $\Gamma = f(W^{(\mathbb{Z})}) $, and the convergence of the operator valued series defined by Equation \ref{e:densite} is ensured by the boundedness of $W^{(\mathbb{Z})}$ and of the Chebychev polynomials ($T_k([-1,1]) \subset [-1,1], \forall k \in \mathbb{Z}$), together with the summability of the sequence 
$(r_k)_{k \in \mathbb{Z}}$.

We will extend usual $MA$ processes to any graph, using this previous remark. This will be the purpose of Section \ref{s:analytic_construction}.

Let us recall some properties about the moving average representation $MA_\infty$ of a process on $\mathbb{Z}$. This representation exists as soon as the $\log$ of the spectral density is integrable (see for instance \cite{davis}).
In this case, there exists a sequence $(a_k)_{k \in \mathbb{N}}$, with $a_0 = 1$, and a Gaussian white noise 
$\mathbf{\epsilon} = (\epsilon_k)_{k \in \mathbb{Z}}.$,
 such that the process $\textbf{X}$ may be written as
$$\forall i \in \mathbb{Z}, X_i = \sum_{k \in \mathbb{N}} a_k \epsilon_{i-k}.$$

Defining the function $h$ over the unit circle $\mathcal{C}$ by
$$\forall x \in \mathcal{C}, h(x) = \sum_{k\in \mathbb{N}} a_k x^k,$$
 we recover, with a few computations, 
the spectral decomposition of the covariance operator $\Gamma$ of $\mathbf{X}$ :

$$\forall i,j \in \mathbb{Z}, \Gamma_{ij} = \int_{[0,2 \pi]} \left|h(e^{it})\right|^2 \mathrm{d}\hat{\mu}_{ij}(t). $$
This implies the equality
$$f\left(\cos(t)\right) = \left|h(e^{it})\right|^2.  $$

Recall that when $h$ is a polynomial of degree $p$ (with non null first coefficient), the process is said to be $MA_p$. In this case, $f$ is also a polynomial of degree $p$. 
Reciprocically, if $f$ is a real polynomial of degree $p$, and as soon as $f\left(\cos(t)\right)$ is even, and non-negative for any $t \in [0,2\pi]$, 
the Fejér-Riesz theorem provides a factorization of $f\left(\cos(t)\right) $ such that $f\left( \cos(t)\right) = \left| h(e^{it})\right|^2$ (see for instance \cite{krein}). This proves that
$\mathbf{X}$ is $MA_p$ if, and only if, its covariance operator may be written $f(W^{(\mathbb{Z})})$, where $f$ is a polynomial of degree $p$.

This remark is fundamental for the construction we provide in the following section (see Definition \ref{def_arma}).

\subsection{Whittle maximum likelihood estimation for time series}

Here, we recall briefly the Whittle's approximation for time series. Let $\Theta$ be a compact interval of $\mathbb{R}^d, d \geq 1$, and $(f_\theta)_{\theta \in \Theta}$ be a parametric family of spectral densities. Let $\theta_0 \in\Theta$, and assume that $(X_i)_{i \in \mathbb{Z}}$ is a Gaussian time series whith spectral density $f_{\theta_0}$.

If we observe $\mathbf{X}_n:= (X_i)_{i = 1, \cdots n}, n>0$, we can define the maximum lokelihood estimate $\hat{\theta}_n$ of $\theta_0$ as:
$$\hat{\theta}_n := \arg \max L_n(\theta,\mathbf{X}_n),  $$
where
$$ L_n (\theta,\mathbf{X}_n) :=-\frac{1}{2} \left( n \log (2 \pi) + \log \det \left(\mathcal{T}_{n}(f_\theta)\right) + \mathbf{X}_{n }^T \big(\mathcal{T}_{n}(f_\theta)\big)^{-1}\mathbf{X}_{n} \right).$$

This estimator is consistent as soon as the spectral densities are regular enough, and under assumptions on the function $\theta \mapsto f_\theta$ (see for instance \cite{AzDa}). However, in practical situations, it is hard to compute. The Whittle's estimate is built by maximizing an approximation of the likelihood instead of the likelihood itself:
$$\tilde{\theta}_n := \arg \max \tilde{L}_n(\theta,\mathbf{X}_n),  $$
where
$$\tilde{ L}_n (\theta,\mathbf{X}_n) :=-\frac{1}{2} \left( n \log (2 \pi) +n \int_{[0,2\pi]} \log  \left(f_\theta(\lambda) \right) \mathrm{d}\lambda + \mathbf{X}_{n }^T \mathcal{T}_{n}(\frac{1}{f_\theta})\mathbf{X}_{n} \right).$$

The Whittle estimate is also consistent and asymptotically normal and efficient, as soon as the spectral densities are regular enough.

The consistency of the Whittle estimate relies on the Szegö's Lemma, which provide a bound on the error between $\frac{1}{n} \log \det \left(\mathcal{T}_{n}(f_\theta)\right)$ and $\int_{[0,2\pi]} \log  \left(f_\theta(\lambda) \right)$. There exists many versions of this Lemma (see for instance \cite{AzDa}, \cite{sego}). 

In this work, we are interested in a weak version given by Azencott and Dacunha-Castelle in \cite{AzDa}. The lemma relies on the following fondamental inequality:
Let $f(x) =\sum_{k \in \mathbb{N}}f_kx^k$ and $g(x) = \sum_{k \in \mathbb{N}} g_k x^k$ be two analytics function on the complex unitar disk. Then we have 
\begin{equation}
\sum_{i,j = 1, \cdots, N} \Bigg| \bigg( \mathcal{T}_N(f)\mathcal{T}_N(g) - \mathcal{T}_N(fg) \bigg)_{ij} \Bigg| \leq \frac{1}{2} \sum_{k \in \mathbb{N}} (k+1)f_k  \sum_{k \in \mathbb{N}} (k+1)g_k.  
\end{equation}

In the following, we aim at developing the same kind of tools for processes indexed by a graph.

% Thanks to this remark, we will build $MA$ processes on any graph $\mathbf{G}$ in the same way. The following section is devoted to this construction, and 
% the last remark actually gives the equivalence between usual $MA_p$ over $\mathbb{Z}$, and the one we
%  will define over any graph $\mathbf{G}$.

%\input{Basic_properties_2011_03_28.tex}
%\input{Analytic_construction_2011_03_29.tex}
\section{Spectral definition of $ARMA$ processes}\label{s:analytic_construction}

In this section, we will define moving average and autoregressive processes over the graph $\mathbf{G}$. 
% In the whole section, we deal with the adjacency
% weighted operator $W$ of the graph $\mathbf{G}$.

% For a more general construction, we allow us to modify the weighted operator $W$ and so the graph $\mathbf{G}$, without 
% adding any edges, as shown in Section \ref{algebraical_construction}. This modifications are, in a sense, isotropic, and leads to stationary 
% operators defined in Section \ref{algebraical_construction}.

As explained in the last section, since $W$ is bounded and self-adjoint, $\operatorname{Sp}(W)$ is a non-empty compact subspace of $\mathbb{R}$,
and $W$ admits a spectral decomposition thanks to an identity resolution $E$, given by
$$W = \int_{\operatorname{Sp}(W)} \lambda \mathrm{d} E (\lambda) .$$

We define here $MA$ and $AR$ Gaussian processes, with respect to the operator $W$, by defining the corresponding classes of covariance operators, since 
the covariance operator fully characterizes any Gaussian process.

\begin{defi} \label{def_arma}
Let $(X_i)_{i \in G}$ be a Gaussian process, indexed by the vertices $G$ of the graph $\mathbf{G}$, and $\Gamma$ its covariance operator.

If there exists an analytic function $f$ defined on the convex hull of $\operatorname{Sp}(W)$, such that
$$\Gamma =\int_{\operatorname{Sp}(W)} f(\lambda) \mathrm{d} E (\lambda), $$ 

we will say that $X$ is
\begin{itemize}
 \item $MA_q$ if $f$ is a polynomial of degree $q$.
\item $AR_p$ if $\frac{1}{f}$ is a polynomial of degree $p$ which has no root in the convex hull of $\operatorname{Sp}(W)$.
\item $ARMA_{p,q}$ if $f = \frac{P}{Q}$ with $P$ a polynomial of degree $p$ and $Q$ a polynomial of degree $q$ with no roots in the convex hull of $\operatorname{Sp}(W)$.
\end{itemize}
Otherwise, we will talk about the $MA_\infty$ representation of the process $\mathbf{X}$.
We call $f$ the \textbf{spectral density} of the process $\mathbf{X}$, and denote its corresponding covariance operator by 
$$\Gamma = \mathcal{K}(f). $$
\end{defi}

\begin{remark}
 Actually, this last construction may also be understood as
$$\Gamma = \mathcal{K}(f) = f(W),$$
in the sense of normal convergence of the associated power series. However, the spectral representation will be useful in the following.
%  (see Section \ref{s:spectral}) and 
Even if we consider only regular processes in this works, the definition using the spectral representation allows weaker regularity than the definition using the normal convergence of the associated power series.
\end{remark}

% The dependency on $W$ is written here because, as said in the introduction of this section, we will in Section \ref{algebraical_construction} provide, for a given
% graph $\mathbf{G}$, other operators we may use for this construction of $ARMA$ processes.
% 
% Actually, we will be allowed to take the discrete Laplacian (defined in section \ref{algebraical_construction}) or the renormalized one. 
% Since we can renormalizes to get all entries in $[-1,1]$, this may be seen as a simple modification of the graph. 

This kind of modeling is interesting when the interactions are locally propagated
(that may be for instance a good modeling for traffic problems.).

The notation $\mathcal{K}(.)$ has to be understood by analogy with the notation $\mathcal{T}(.)$ used for Toeplitz operators.

Notice %also
 that, in the usual case of $\mathbb{Z}$, and for finite order $ARMA$, we recover the usual definition as shown
 in Subsection \ref{ss:maforz}. So, the last definition may be seen as an extension of isotropic $ARMA$ for any graph $\mathbf{G}$.
Besides, note that this extension is given by the equivalence, for any $g \in \mathbb{L}^2\left([0, 2\pi]\right)$, such that $\int_{[0, 2\pi]} \log(g) <+ \infty$,
$$\forall f \in \mathbb{L}^2([-1,1]), \left(g = f\left(\cos(t)\right) \Leftrightarrow  \mathcal{T}(g) = \mathcal{K}(f)\right). $$
This means that, in the usual case $\mathbf{G}=\mathbb{Z}$, the definition of spectral density in our framework is the usual one, up to an change of variable $\lambda = \cos(t)$ (see Subection \ref{ss:maforz}).

Now, we get a representation of moving average processes over any graph $\mathbf{G}$. The following section gives the main result of this paper. It deals with 
the maximum likelihood identification.
% , and Section \ref{s:algebraical_construction} proposes a general definition of stationary processes indexed by the vertices of
% a graph, and shows the stationarity of $MA$ representation thanks to this definition. We 
% underline here that the definition given in Section \ref{s:algebraical_construction} allows constructions of $MA$ processes built with some ``isotropic modifications'' of $W$.

\section{Convergence of  maximum approximated likelihood estimators}\label{s:main_thm}
In this section as before, $\mathbf{G} = (G,W)$ is a graph with bounded degree.
% ($\operatorname{deg}(\mathbf{G}) < + \infty$)
Let also
$(X_i)_{i \in G}$ be a Gaussian spatial process indexed by the vertices of $\mathbf{G}$ with spectral density $f_{\theta_0}$ (defined in Section 
\ref{s:analytic_construction}) depending on an unknown parameter
$\theta_0 \in \Theta$. We aim at estimating $\theta_0$. For this, we will generalize classical maximum likelihood estimation of time series.
 
We will also develop a Whittle's approximation for $ARMA$ processes indexed by the vertices of a graph. 
% Instead of maximizing the likelihood, we use an approx
% 
% approximation facilitates the computation
% That is an approximation of the 
% likelihood that provides convergence of likelihood estimate. 
We follow here the guidelines of the proof given in \cite{AzDa} for
the usual case of time series.

\subsection{Framework and Assumptions}
\label{ss:framework}
Let us now specify the framework of our study. Let $(\mathbf{G}_n)_{ n \in \mathbb{N}}$ be a growing 
sequence of finite nested subgraphs. This means that if $\mathbf{G}_n = (G_n,W_n)$, we have $G_n \subset G_{n+1} \subset G$ and that for any $i,j \in G_n$, it holds that  $W_n(i,j) = W(i,j)$. 

% We will use the following notations. 
Let $m_n = \operatorname{Card}(G_n)$. 
We set also
  $$\delta_n = \text{Card}\left\{i \in G_n, \exists j \in G \backslash G_n, W_{ij} \neq 0 \right\}.$$ 

The sequence $(m_n)_{n \in \mathbb{Z}}$ may actually be seen as the ``volume'' of the graph $\mathbf{G}_n$, and $\delta_n$ as the size of the 
boundary of $G_n$. 
For the special case $G = \mathbb{Z}^d$ and $G_n = [-n,n]^d$, we get $m_n = (2n+1)^d$ and $\delta_n = 2d(2n+1)^{d-1}$.

The ratio $\frac{\delta_n}{m_n}$ is a natural quantity associated to the expansion of the graph that also appears 
in isoperimetrical \cite{isoper} and graph expander issues.
We will assume here that this ratio goes to $0$ when the size of the graph goes to infinity. In short, we set 
\begin{hyp}\label{hyp_graph_1} 
$~~  $
% \begin{itemize}
%  \item $\sup_{i,j \in G} W_{ij} \leq \frac{1}{D} $
 $\delta_n = o(m_n)$
% \end{itemize}
\end{hyp}

% Notice here that the first assumption implies that $\operatorname{Sp}(W) \subset \left[-1,1\right]$.
This assumption is a non-expansion criterion. The graph has to be amenable, which is satisfied for the last examples $G = \mathbb{Z}^d$ and $G_n = [-n,n]^d$,
 but not for a homogeneous tree, whatever the choice of the sequence of subgraphs $(\mathbf{G_n})_{n \in\mathbb{N}}$ is.

We will now choose a parametric family of covariance operators of $MA$ processes as defined in the last section. 
First, let $\Theta$ be a compact interval
of $\mathbb{R}$. 

We point out that for sake of simplicity, we choose a one-dimensional parameter space $\Theta$. Nevertheless, all the results could be easily extended to the case $\Theta \subset \mathbb{R}^k, k \geq 1$.

Define $\mathcal{F}$ as the set of positive analytic functions over the convex hull of $\operatorname{Sp}(W)$.

Let also $(f_\theta)_{ \theta \in \Theta}$ be a parametric family of functions of $\mathcal{F}$.
They define a parametric set of covariances on $G$ (see Section \ref{s:analytic_construction}) by 
$$\mathcal{K}(f_\theta) = f_\theta(W).$$

As in \cite{AzDa}, we will need a strong regularity for this family of spectral densities. 

Let us introduce a regularity factor for any analytic function 
$$f \in \mathcal{F}, f(x) = \sum_k f_k x^k \left(x \in \operatorname{Sp}(W)\right),$$
 by setting
\begin{equation}
\alpha(f) := \sum_{k \in \mathbb{N}} \left| f_k \right| (k+1).  
\end{equation}

Now, let $\rho>0$ and define,
\begin{equation}
\mathcal{F}_\rho := \left\{ f \in \mathcal{F} ,\alpha(\log(f)) \leq \rho\right\}. 
\end{equation}
Notice that for any $f\in \mathcal{F}_\rho$, we have $\alpha(f) \leq e^{\rho}, \alpha(\frac{1}{f}) \leq e^{\rho}$.

We need the following assumption
\begin{hyp}\label{hyp_function}$ $

% We make the following assumptions :
\begin{itemize}
 \item The map $\theta \rightarrow f_\theta$ is injective.
\item For any $ \lambda \in \text{Sp}(W)$, the map $\theta \rightarrow f_\theta(\lambda)$ is continuous.
\item $\forall \theta \in \Theta, f_\theta \in \mathcal{F}_\rho $ .
\end{itemize}
\end{hyp}

From now on, consider $\theta_0 \in \mathring{\Theta}$. Let $\mathbf{X}$ be a centered Gaussian $MA_\infty$ process over $\mathbf{G}$ 
with covariance operator $\mathcal{K}(f_{\theta_0})$ (see Section \ref{s:analytic_construction}). 

We observe the restriction of this process on the subgraph $\mathbf{G}_n$ defined before. Our aim is to compute 
the maximum likelihood estimator of $\theta_0$.
Let $X_n = (\textbf{X}_i)_{i\in G_n}$ be the observed process and $\mathcal{K}_n(f_\theta)$ be its 
covariance :
$$X_n \sim \mathcal{N}\left(0,\mathcal{K}_n(f_{\theta_0}) \right).$$

The corresponding log-likelihood at $\theta$ is

$$L_n (\theta) :=-\frac{1}{2} \left( m_n \log (2 \pi) + \log \det \left(\mathcal{K}_{n}(f_\theta)\right) + X_{n }^T \big(\mathcal{K}_{n}(f_\theta)\big)^{-1}X_{n} \right).$$

% Recall what happens in the usual case of time series ($G = \mathbb{Z}$, $G_n = [-n,n]$). 
% Using the notations of Section \ref{s:rappels}, we assume that $\mathbf{X}$ is a stationary process indexed by $\mathbb{Z}$ with covariance operator $\Gamma$.
% Recall that, as explained in Section \ref{s:analytic_construction}, if the usual spectral density $g \in \mathbb{L}^2([0,2\pi])$ is such
% that $\log(g)$ is integrable, then the function $f \in \mathbb{L}^2([-1,1])$ defined by $g(t) = f\left(\cos(t)\right)$, is the spectral density of our framework.
% That means that 
% $$\Gamma = \mathcal{K}(f) = f(W^{(\mathbb{Z})} ) = \mathcal{T}(g). $$

As discussed before, in the case $G =  \mathbb{Z}$, it is usual to maximize an approximation of the likelihood. The classical approximation is the Whittle's one (\cite{sego}), 
where 
$$\frac{1}{n}\log \det \left(\mathcal{T}_{n}(g)\right)$$
 is replaced by 
$$\frac{1}{2\pi}\int_{[0,2\pi]} \log \left(g\left(t\right)\right) \mathrm{d}t .$$

Back to the general case, we aim at performing the same kind of approximation. For this, we will need the following assumption
to ensure the convergence of $\log \det \left(\mathcal{K}_{n}(f_\theta)\right)$ (see Section~\ref{s:rappels} 
for the definition of $\mu_{ii}$) :

\begin{hyp}\label{hyp_graph_2} 
There exists a positive measure $\mu$, such that
 $$ \frac{1}{m_n} \sum_{i \in G_n} \mu_{ii} \underset{n \rightarrow \infty}{\xrightarrow{\mathcal{D} }} \mu.$$
Here, $\mathcal{D}$ stands for the convergence in distribution
\end{hyp}

The limit measure $\mu$ is classically called the spectral measure of $\mathbf{G}$ with respect to the sequence of subgraphs $(\mathbf{G}_n)_{n \in \mathbb{Z}}$ (see \cite{survey_graph} for example).
% In Section \ref{s:spectral}, we give some sufficient conditions to ensure the existence of the spectral measure.

Actually, under Assumption \ref{hyp_graph_1}, Assumption \ref{hyp_graph_2} is equivalent to the convergence of the empirical distribution of eigenvalues of $W_{G_n}$ (here, $W_{G_n}$ denotes the restriction of $W$ over the subgraph $G_n$)
%  (see Section \ref{s:main_thm}), and 
That is, if 
$\lambda^{(n)}_1,\cdots, \lambda^{(n)}_{m_n}$ denote the 
 eigenvalues (written with their multiplicity orders) of $W_{g_n}$, 
Define $$\mu^{[1]}_n := \frac{1}{m_n} \sum_{i = 1}^{m_n} \delta_{\lambda^{(n)}_i}, $$
and
$$\mu^{[2]}_n =  \frac{1}{m_n} \sum_{i \in G_n} \mu_{ii}, $$
% where the measure $\mu_{ii}$ is defined in Section \ref{s:rappels}. Define also

Then, under Assumption \ref{hyp_graph_1}, the convergence of $\mu^{[1]}_n$ to $\mu$ (i.e. Assumption \ref{hyp_graph_2}) is equivalent to the convergence of $\mu^{[2]}_n$ to $\mu$.
% Assumption \ref{hyp_graph_2} deals with the convergence of $\mu^{[1]}_n$ to the spectral measure $\mu$. 
%

%  Actually, under Assumption \ref{hyp_graph_1}, this last assumption is equivalent to the following. 
% 
% \begin{hyp}{Weak convergence of the spectral measure}
% \label{a:mes}
%  $$\mu^{[1]}_n  \underset{n \rightarrow \infty}{\rightarrow} \mu.$$
% in the sense of the weak convergence.
%  \end{hyp}

To prove this equivalence 
% between Assumption \ref{hyp_graph_2} and Assumption \ref{a:mes}, 
we just have to notice that :
\begin{eqnarray*}
\int_{\operatorname{Sp}(W)} \lambda^k \mathrm{d}\mu^{(1)}_n(\lambda)& -& \int_{\operatorname{Sp}(W)} \lambda^k \mathrm{d}\mu^{(2)}_n(\lambda)
\\& =& \frac{1}{m_n} \sum_{i = 1}^{m_n} \left(\lambda^{(n)}\right)_i^k - \frac{1}{m_n} \sum_{i \in G_n} (W^k)_{ii}
\\ & = & \frac{1}{m_n}\operatorname{Tr}\left((W_{G_n})^k\right)- \frac{1}{m_n}\operatorname{Tr}\left((W^k)_{G_n}\right).
\end{eqnarray*}

So that, we get the result by Lemma \ref{lem_hom} (see Section \ref{s:szego}).

%FIXME : Further details are given

As in the case of time series (for $ G = \mathbb{Z}$), we can approximate the log-likelihood. 
It avoids
an inversion of a matrix and a computation of a determinant. Indeed, we will consider the two following approximations.

$$\bar{L}_n (\theta) := -\frac{1}{2}\left(m_n \log (2 \pi) + m_n \int \log(f_\theta(x)) \mathrm{d}\mu(x) + 
X_{n }^T \left(\mathcal{K}_{n}(f_\theta)\right)^{-1}X_{n} \right).$$

$$\tilde{L}_n (\theta) := -\frac{1}{2}\left(m_n \log (2 \pi) + 
m_n \int \log(f_\theta(x)) \mathrm{d}\mu(x) + X_{n }^T \left(\mathcal{K}_{n}\left(\frac{1}{f_\theta}\right)\right)X_{n} \right).$$

Notice that approximated maximum likelihood estimators are not asymptotically normal in general (see for instance \cite{guyonart} for $\mathbb{Z}^d$). Indeed, the 
score associated to the approximated $\log$-likelihood has to be
asymptotically unbiased \cite{AzDa}.  

To overcome this problem in $\mathbb{Z}^d$, the tapered periodogram can be used (see \cite{guyonbook}, \cite{guyonart}, \cite{dahlhaus}). 

Let us consider graph extensions of standard time series models : 

% This is tractable only in two cases :
\begin{itemize}
 \item \underline{The $MA_P$ case :} There exists $P>0$ such that the true spectral density $f_{\theta_0}$ is a polynomial of degree bounded by $P$. 
% \textit{We will refer to this case as the $MA_K$ case}.
\item \underline{The $AR_P$ case :} There exists $P>0$ such that \textit{all} the spectral densities (for any $\theta \in \Theta$) of the parametric set are such that 
$\frac{1}{f_\theta}$ is a polynomial of degree bounded by $P$.
% \textit{We will refer to this case as the $AR_K$ case}.
\end{itemize}

So, to define the good approximated $\log$-likelihood, we first introduce the unbiased periodogram in each of the last cases. 
Now, let $P>0$.

% We first count, for any couple of vertices $(i,j) \in G^2$,  the number of path (counted with their weight) 
% going from $i$ to $j$ with length $p = 0, \cdots, P$.
% 
% For this, define the $P$-type of a couple of vertex $(i,j) \in G^2$ as the $(P+1)$-tuple given by
% $$t(i,j) := \left(W^{(p)}_{ij}\right)_{0 \leq p \leq P}.$$

% Then, 
Define a subset $V_P$ of signed measures on $\mathbb{R}$ as
 $$V_P: = \left\{ \mu_{ij}, i,j \in G , d_{\mathbf{G}}(i,j) \leq P  \right\},$$ 
where $d_{\mathbf{G}}(i,j), i,j \in G$ stands for the usual distance on the graph $\mathbf{G}$, i.e. the length of the shortest path going from $i$ to $j$.

% This set gives all the possible local measures $\mu_{ij}, i,j \in G$ (see Section \ref{s:rappels}) over the graph $\mathbf{G}$.
% This set gives among any $i,j \in G$ all possible $(P+1)$-tuple of the number of path (counted with their weight) 
% going from $i$ to $j$ with length $p = 0, \cdots, P$. 

% We highlight that it is finite as soon as the entries of the weighted adjacency operator $W$ lie in a finite set, 
% and since the degree of $\mathbf{G}$ is bounded. 

% In the following, we choose $n$ large enough to ensure that 
We will need the following assumption
\begin{hyp}\label{a:graph_structure}
The set $V_P$ of possible local measures over $G$ is finite, and $n$ is large enough to ensure that
$$\forall v \in V_P, \exists (i,j) \in G_n^2, \mu_{ij} = v.$$
\end{hyp}

% This assumption ensure that $G_n$ is large enough to contains at least one couple of vertices of any $P$-type which appears in the graph $G$.

% This is possible since $V_P$ is finite, and since $t(i,j) = \left\{0\right\}^{P+1}$ as soon as the shortest path from $i$ to $j$ is longer than $P$. 
\begin{remark}
This assumption is quite strong, and holds for instance for quasi-transitive graphs (i.e. such that the quotient of the graph with its automorphism group is finite).
%  built by reproducing a finite graph (the pattern) at each vertex of a distance transitive graph.
% In particular, this holds for any Cayley graph or distance transitive graph. 
This assumption may be relaxed, but it is a hard and technical work that will be the issue of a forthcoming paper. 
\end{remark}

Define now the matrix $B^{(n)}$ (the dependency on $P$ is omitted, for clarity) by 
\begin{eqnarray*} B^{(n)}_{ij}  &:=& \frac{\text{Card}\left\{(k,l) \in G_n\times G, \mu_{kl}= \mu_{ij} \right\}       }
{ \text{Card}\left\{(k,l) \in G_n\times G_n, \mu_{kl}= \mu_{ij} \right\}  }, \text{ if } ,d_{\mathbf{G}}(k,l)\leq P 
\\ & := & 1 \text{ if } d_{\mathbf{G}}(k,l)> P.
\end{eqnarray*}

The matrix $B^{(n)}$ gives a boundary correction, comparing, for any $v\in V_P$ the frequency of the interior
 couples of vertices with local measure $v$ with the boundary couples of vertices with local measure $v$.
Actually, this way to deal with the edge effect is very similar to the one used for $\mathbf{G} =\mathbb{Z}^d$ (see \cite{dahlhaus}, \cite{guyonart}).

% The explicit construction of $B^{(n)}$ is given at the end of the section for $G = \mathbb{Z}^2$ and $P=2$. 

As example, let us now describe the case $G = \mathbb{Z}^2$, for $P=2$. In this case $W^{(\mathbb{Z}^2)}$ is
$$\forall i,j,k,l \in \mathbb{Z}, W^{(\mathbb{Z}^2)}\left((i,j),(k,l)\right) := \frac{1}{4} \ind_{\left|i-j\right|+\left|k-l\right| = 1} .$$

In this example, we set $G_n = [1,n]^2$, and we can compute the matrix $B^{(n)}$. Indeed, it only is needed to notice that 
$$\mu_{(i_1,j_1),(i_1+k,j_1+l)} = \mu_{(i_2,j_2),(i_2 + \epsilon_1 k,j_2 + \epsilon_2 l)}, i_1,i_2,j_1,j_2,k,l \in \mathbb{Z}, \epsilon_1, \epsilon_2 \in \left\{-1,1\right\}.$$
This means that the local measure of a couple of vertices depends only of their relative positions (stationarity and isotropy of this set of measure).  
So, we need to count the configurations given by Figure \ref{type} since we consider only couples of vertices $u,v \in \mathbb{Z}^2$ such that $d_{\mathbb{Z}^2}(u,v) \leq 2$.

\begin{figure}[htbp]
 \centering
 \includegraphics[bb=14 14 535 455,scale=0.4,keepaspectratio=true]{./type.eps}
 % type.pdf: 521x441 pixel, 72dpi, 18.38x15.56 cm, bb=0 0 521 441
 \caption{Possible configurations for couple of vertices}
 \label{type}
\end{figure}

% We denote by $u,v$ the vertices of $\mathbb{Z}^2$ and the canonical generator of $\mathbb{Z}^2$ by $e_1,e_2$. 
% This means that if $u = (i,j), i,j \in \mathbb{Z}$, then we can write $v= u + ke_1 +le_2, k,l \in \mathbb{Z}$ for $v = (i+k,j+l)$.
% % If $u = (i,j)$, we write $u = ie_1+je_2$.
% Actually there exists only five $2$-type of vertices (see Figure \ref{type} for the meaning of this construction) :
% $\forall u \in \mathbb{Z}^2$
% % , u = (i,j)$
% \begin{eqnarray*}
%  t_1 &:= \Big(0,0,\frac{1}{4}\Big)&= \mu_{uu}
% \\ t_2 &:= \Big( 0,1,0 \Big)& = \mu_{u,u+e_1}
%  \\    &                       &=   t(u,u-e_1) 
%  \\    &                       &=   t(u,u+e_2) 
%  \\    &                       &=   t(u,u-e_2)
% \\t_3 &:= \Big(0,0,\frac{1}{8} \Big)&= t(u,u+e_1+e_2)
%  \\    &                       &=      t(u,u+e_1-e_2)
%  \\    &                       &=    t(u,u-e_2+e_2)
%  \\    &                       &=   t(u,u-e_1-e_2) 
% \\t_4 &:=   \Big( 0,0,\frac{1}{16} \Big) & =  t(u,u+2e_1) 
%  \\    &                       &=   t(u,u-2e_1) 
%  \\    &                       &=   t(u,u+2e_2)
%  \\    &                       &=   t(u,u-2e_2) 
% \\ t_5 &:= \Big(0,0,0 \Big)&=t(u,v)  \text{ in all the other cases } 
% \end{eqnarray*}

We get, for any $i,j \in \mathbb{Z}$,
\begin{itemize}
 \item $B^{(n)}_{(i,j),(i,j)} = \frac{n^2}{n^2}= 1 .$
 \item $B^{(n)}_{(i,j),(i,j \pm 1)} = B^{(n)}_{(i,j),(i \pm 1,j)} = \frac{4n(n-1)}{4n^2}. $
\item $B^{(n)}_{(i,j),(i \pm 1,j \pm 1)}= \frac{4(n-1)^2}{n^2}. $
\item $B^{(n)}_{(i,j),(i,j\pm 2)}= B^{(n)}_{(i,j),(i \pm 2,j)}= \frac{4n(n-2)}{4n^2}  $
\end{itemize}

One can notice that 
 $$\sup_{ij} \left| B^{(n)}_{ij}-1\right| \underset{n \rightarrow \infty}{\rightarrow} 0.$$
Assumption \ref{hyp_unbiased} ensure that this property holds for the graph we consider.

% FIX ME : Example ?

Back to the general case, let $f\in \mathcal{F}_\rho$. We define the unbiased periodogram as 
$$X_n^T \mathcal{Q}_n(\frac{1}{f}) X_n. $$
where
$$\mathcal{Q}_n(f) := B^{(n)}\odot \mathcal{K}_n(f) .$$
Here, the operation $\odot $ denotes the Hadamard product for matrices, that is 
$$ \forall i,j \in G_n, \left(B^{(n)}\odot \mathcal{K}_n(f)\right)_{ij} = \left(B^{(n)} \right)_{ij}{\mathcal{K}_n(f)}_{ij}.$$
Notice that this is actually a way to extend the so called tapered periodogram (see for instance \cite{guyonart}).

We now define the unbiased empirical log-likelihood, for any $\theta \in \Theta$
$$ L^{(u)}_n (\theta) := -\frac{1}{2}\left(m_n \log (2 \pi) + 
m_n \int \log(f_\theta(x)) \mathrm{d}\mu(x) + X_{n }^T \left(\mathcal{Q}_{n}(\frac{1}{f_\theta})\right)X_{n} \right). $$

We denote by $\hat{\theta}_n$, $\tilde{\theta}_n$, $\bar{\theta}_n$, $\theta^{(u)}$ the maximum likelihood estimators associated to $L_n$, $\tilde{L}_n$, $\bar{L}_n$, $L^{(u)}_n$, respectively.

We will need the following assumption,
\begin{hyp}\label{hyp_unbiased}
There exists a positive sequence $(u_n)_{n \in \mathbb{N}}$ such that, 
$$ u_n \underset{n \rightarrow \infty}{\rightarrow} 0,$$
 and
$$ \sup_{ij} \left| B^{(n)}_{ij}-1\right| \leq u_n.$$
\end{hyp}

Notice that the last assumption holds for example in the case $\mathbf{G} = \mathbb{Z}^d,d>1$.

To prove asymptotic normality and efficiency 
of the estimator $\theta_n^{(u)}$, we will also need the following assumption.
\begin{hyp}\label{hyp_norm}
Assume that
\begin{itemize}
 \item There exists a positive sequence $(v_n)_{n\in \mathbb{N}} $ such that $v_n = o(\frac{1}{\sqrt{m_n}})$ and

$$\forall f \in \mathcal{F}_\rho, \left| \frac{1}{m_n} \operatorname{Tr}(\mathcal{K}_{G_n}(f)) - \int  f \mathrm{d}\mu \right| \leq \alpha(f) v_n.$$

\item For any $\theta \in \Theta$, $f_\theta$ is twice differentiable on $\Theta$ and 
$$\frac{\mathrm{d}}{\mathrm{d}\theta}(f_\theta) \in \mathcal{F}_\rho, \frac{\mathrm{d}^2}{\mathrm{d}\theta^2}(f_\theta)\in \mathcal{F}_\rho.$$
\end{itemize}
\end{hyp}

The first assumption means that the convergence of the empirical distribution of eigenvalues of $\mathcal{K}(f)$ to the spectral measure $\mu$ is faster than $\frac{1}{\sqrt{m_n}}$. It holds for instance for quasi-transitives graphs, with a suitable sequence of subgraphs. The second assumption is more classical. For example it is required in the case $\mathbf{G} = \mathbb{Z}$ (see \cite{AzDa}).
% and same as in the usual case of $\mathbb{Z}$ (see \cite{AzDa}).

% 

\subsection{Convergence and asymptotic optimality}

Let $\rho>0$.
We can now state one of our main result:
\begin{thm}\label{thm_conv}
Under Assumptions \ref{hyp_graph_1}, \ref{hyp_function} and \ref{hyp_graph_2}, the sequences $(\hat{\theta}_n)_{n \in \mathbb{N}}$, 
$(\bar{\theta}_n)_{n \in \mathbb{N}}$, $(\tilde{\theta}_n)_{n \in \mathbb{N}}$ 
converge, as $n$ goes to infinity, $P_{f_{\theta_0}}$-a.s. to the true value $\theta_0$.
If moreover Assumption \ref{hyp_unbiased} holds, this is also true for $(\theta^{(u)}_n)_{n \in \mathbb{N}}$.
\end{thm}

\begin{proof}
The proof follows the guidelines of \cite{AzDa}. We highlight the main changes performed here.
First, we define the Kullback information on $G_n$ of $f_{\theta_0}$ with respect to $f \in \mathcal{F}_\rho$, by

$$\mathbb{IK}_n(f_{\theta_0},f) := \mathbb{E}_{P_{f_{\theta_0}}}\left[ - \log(\frac{\mathrm{d}P_{f}}{\mathrm{d}P_{f_{\theta_0}}})\right]. $$

and the asymptotic Kullback information (on $\mathbf{G}$) by

$$\mathbb{IK}(f_{\theta_0},f) = \lim_n \frac{1}{m_n} \mathbb{IK}_n(f_{\theta_0},f)$$ whenever it is finite.

The convergence of the estimators of the maximum approximated likelihood is a direct consequence of the following lemmas :

\begin{lem}\label{lem_ik}
For any $f\in \mathcal{F}_\rho$, and under Assumptions \ref{hyp_graph_1},  \ref{hyp_function} and \ref{hyp_graph_2}, the asymptotic Kullback information exists and may be written as 
$$\mathbb{IK}(f_{\theta_0},f) = \frac{1}{2}\int \left(-\log(\frac{f_{\theta_0}}{f}) - 1 + \frac{f_{\theta_0}}{f}\right) \mathrm{d} \mu . $$ 
Furthermore, if we set $l_n(\theta,X_n) = \frac{1}{m_n}L_n(\theta,X_n)$, we have that $P_{f_{\theta_0}}$-a.s.,
 $$ l_n(\theta_0,X_n) - l_n( \theta,X_n) \underset{n \rightarrow \infty}{\rightarrow} \mathbb{IK}(f_{\theta_0},f_\theta) %= \frac{1}{2}\int -\log(\frac{f_{\theta_0}}{f_\theta}) - 1 + \frac{f_{\theta_0}}{f_\theta} \mathrm{d} \mu
$$
uniformly in $\theta \in \Theta$.

This property also holds for $\bar{l}_n := \frac{1}{m_n}\bar{L}_n$ and $\tilde{l}_n:= \frac{1}{m_n}\tilde{L}_n$ 

Furthermore, for $P>0$, and for both the $AR_P$ or the $MA_P$ case (see above), this also holds for $l_n^{(u)} := \frac{1}{m_n}L^{(u)}_n$.
\end{lem}

\begin{lem}\label{lem_conv}
Let $f_{\theta_0}$ be the true spectral density, and $(\ell_n)_{n \in \mathbb{N}}$ be a deterministic sequence of continuous functions such that 
 $$ \forall \theta \in \Theta, \ell_n(\theta_0) - \ell_n(\theta) \underset{n \rightarrow \infty}{\rightarrow} \mathbb{IK}(f_{\theta_0},f_\theta)  $$
uniformly as $n$ tends to infinity.
Then, if $\theta_n = \arg \max_\theta \ell_n(\theta)$, we have
$$\theta_n \underset{n \rightarrow \infty}{\rightarrow} \theta_0 .$$
\end{lem}

The proofs of these lemmas are postponed in Appendix (Subsection \ref{ss:proof_lemmas_main_thm}).
\end{proof}

\begin{thm}\label{t:as_norm}
In both the $AR_P$ or $MA_P$ cases, and 
and under all previous assumptions
\ref{hyp_graph_1},  \ref{hyp_function}, \ref{hyp_graph_2}, \ref{a:graph_structure}, \ref{hyp_unbiased}, \ref{hyp_norm}, the estimator $\theta^{(u)}_n$ of $\theta_0$ is asymptotically normal:
$$\sqrt{m_n}(\theta^{(u)}_n - \theta_0 ) \underset{n \rightarrow \infty}{ \xrightarrow{\mathcal{D}} } \mathcal{N}\Bigg(0,\left(\frac{1}{2}\int \left(\frac{f'_{\theta_0}}{f_{\theta_0}}\right)^2\mathrm{d}\mu \right)^{-1}\Bigg) .$$

Furthermore, the Fisher information of the model is 
$$ J(\theta_0):= \frac{1}{2}\int \left(\frac{f'_{\theta_0}}{f_{\theta_0}}\right)^2\mathrm{d}\mu .$$
Hence, the previous estimator is asymptoticly efficient. 
 \end{thm}

\begin{proof}
Here again, we mimic the usual proof by extending the result of \cite{AzDa} to the graph case.

Using a Taylor expansion, we get 
$$(l_n^{(u)})'(\theta_0) = (l_n^{(u)})'(\theta_n^{(u)}) + (\theta_0 - \theta_n^{(u)})(l_n^{(u)})''(\breve{\theta}_n)  ,$$
where $\breve{\theta}_n \in \left]\theta^{(u)}_n, \theta_0\right[.$
As $\theta_n^{(u)} = \arg \max l_n^{(u)}$, we have 
$$(l_n^{(u)})'(\theta_n^{(u)}) = 0.$$
So that,
$$\sqrt{m_n}(\theta_0 - \theta_n^{(u)}) = \left((l_n^{(u)})''(\breve{\theta}_n)\right)^{-1}\sqrt{m_n}(l_n^{(u)})'(\theta_0).$$
The end of the proof relies on three lemmas : 

Lemma \ref{norm} provides the asymptotic normality for $\sqrt{m_n}(l_n^{(u)})'(\theta_0) $. Combined with Lemma \ref{var}, we get the asymptotic normality 
for  $\sqrt{m_n}(\theta_0 - \theta_n^{(u)})$.
Finally, Lemma \ref{fi_info} gives the Fisher information.
%  and this proves that the estimator $\theta_n^{(u)}$ realize the Cramér-Rao bound. This provides 
% the efficiency of this estimator.
\begin{lem}\label{norm}
 $$ \sqrt{m_n}(l_n^{(u)})'(\theta_0) \underset{n \rightarrow \infty}{\xrightarrow{\mathcal{D}}} \mathcal{N}\bigg(0,\frac{1}{2}\int\left( \frac{f'_{\theta_0}}{f_{\theta_0}}\right)^2\mathrm{d}\mu\bigg).$$
\end{lem}

\begin{lem}\label{var}
 $$\left((l_n^{(u)})''(\breve{\theta}_n)\right)^{-1} \underset{n\rightarrow \infty}{\rightarrow} 2\left(\int\left(\frac{f'_{\theta_0}}{f_{\theta_0}}\right)^2\mathrm{d}\mu \right)^{-1}, P_{f_{\theta_0}}-\text{ a.s.} $$
\end{lem}

\begin{lem}\label{fi_info}
 The asymptotic Fisher information is :
$$J(\theta_0) = \frac{1}{2}\int \left(\frac{f'_{\theta_0}}{f_{\theta_0}}\right)^2\mathrm{d}\mu.$$
\end{lem}
The proofs of these lemmas are postponed in Appendix (Subsection \ref{s:technical_lemmas}) 

\end{proof}

\section{Discussion}
\label{s:discussion}
% In this section, we discuss some potentials applications of our results and some perspectives. 

Note first that Theorem \ref{thm_conv} provides consistency of the estimators under weak conditions on the graph. Indeed, amenability ensures Assumption \ref{hyp_graph_1}, for a suitable sequence of subgraphs. Assumption \ref{hyp_graph_2} holds as soon as there is a kind of homogeneity in the graph. The simplest application is quasi-transitives graph. Note that if $\mathbf{G}$ is ``close'' to be quasi-transitive, Assumption \ref{hyp_graph_2} is still true. We also could adapt notions of unimodularity \cite{aldous} or stationarity \cite{curien} to our framework and prove the existence of a spectral measure. Furthermore, Assumption \ref{hyp_graph_2} holds for the real traffic network (this will be explained in a forthcomming paper).

% All this previous work may seem quite formal, but in practical situations, it may be applied to many examples. Here we underline two very special cases.
To build the estimator $\theta_n^{(u)}$, stronger assumptions on the graph $\mathbf{G}$ are needed. Let us discuss two very special cases.
First, Theorem \ref{t:as_norm} may be applied in the $\mathbb{Z}^d$ case with holes, that is in the presence of missing data, up to the condition that they remain few enough. Actually, Assumption \ref{hyp_graph_1} is required, so the boundary of the subgraphs (counting the holes) has to be small in front of the volume of this subgraphs. 

We need furthermore a kind of homogeneity for these holes. For instance, we can assume that the data are missing completely at random. This particular case is interesting for prediction issues.
% , because we have in this case to estimate the covariance structure of the process observed with missing values.\vskip .1in

Another strong potential application is quasi-transitive graphs, as mentioned above. Indeed, take for instance a finite graph (the pattern) and reproduce it at each vertex of an infinite (amenable) vertex-transitive graph. The final graph is then quasi-transitive, and all the previous assumptions hold. 

This seems to be a natural extension of what happens for $\mathbb{Z}^d$. Furthermore, in this situation as in $\mathbb{Z}^d$, our work may also be applied to a process with missing values.\vskip .1in

Note also that conditions of both amenability of the graphs and regularity of spectral densities seem natural, looking at the Szegö's Lemmas (see Section \ref{s:szego}). Indeed, the difference computed in Lemma \ref{lem_hom} is only due to edge effects. \vskip .1in

Thus, there are two ways for relaxing this conditions.
On the one hand, it could be interesting to deal with lower regularity (for instance to study long memory processes) for the spectral densities. On the other hand, it could be also interesting to relax conditions on the graph, for instance for more regular densities. In particular, we could investigate the case of random graphs, and try to pick up homogeneity conditions into the random structure. 
As mentioned above, another natural extension of this work could be done to graphs ``close'' to be quasi-transitive. 

These two limits of our present work are actually two of our main perspectives in this framework.

\section{Simulations} \label{s:simul}
In this section, we give some simulations over a very simple case, where the graph $G$ is built taking some rhombus 
connected by a simple edge both on the left and right (see Figure \ref{f:graphtest}).

\begin{figure}[htbp]
\caption{Graph $G$}
 \centering
 \includegraphics[bb=2 0 471 33,scale=0.7,keepaspectratio=true]{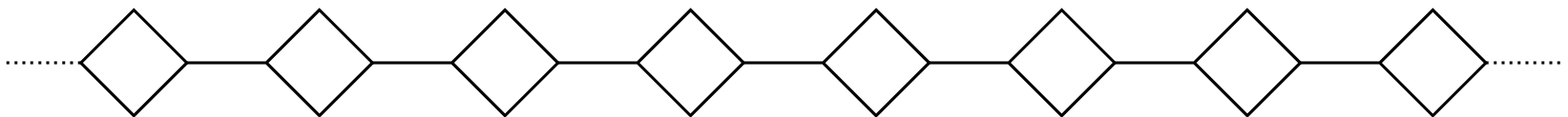}
 % graph_ex2.eps: 0x0 pixel, 300dpi, 0.00x0.00 cm, bb=2 0 471 33
   \label{f:graphtest}
\end{figure}

The sequence of nested subgraphs chosen here is the growing neighborhood sequence (we chose a point $x$ and we take $G_n = \left\{y \in G, d_{\mathbf{G}}(x,y)\leq n\right\}$).
We study an $\text{AR}_2$ model, where, 
\begin{eqnarray*}
 \Theta  = \left] -1,1 \right[,
 \\  f_\theta(x) = \left(\frac{1}{1-\theta x}\right)^2 ( \theta \in \Theta).
\end{eqnarray*}

Here, we take for $W$ the adjacency operator of $G$ normalized in order to get $\sup_{i,j \in G} W_{ij} \leq \frac{1}{\operatorname{deg}(G)}$.
We choose $\theta_0 = \frac{1}{2}$, $m_n = 724$.
We approximate the spectral measure of $G$ by the spectral measure of a very large graph (around $10000$ vertices) built in the same way.
 Figure \ref{f:spectest}
 shows the empirical spectrum of the graph $G$ with respect to the
 sequence of subgraphs $(G_n)_{n \in \mathbb{N}}$.

%************FIXME : Include graphique spectrum !!!!!!!!!!!!!!!!!!!!!!!!!!!!!!!! ******************

\begin{figure}[htbp]
\caption{Empirical spectrum}
 \centering
 \includegraphics[bb=0 0 610 460,scale=0.5,keepaspectratio=true]{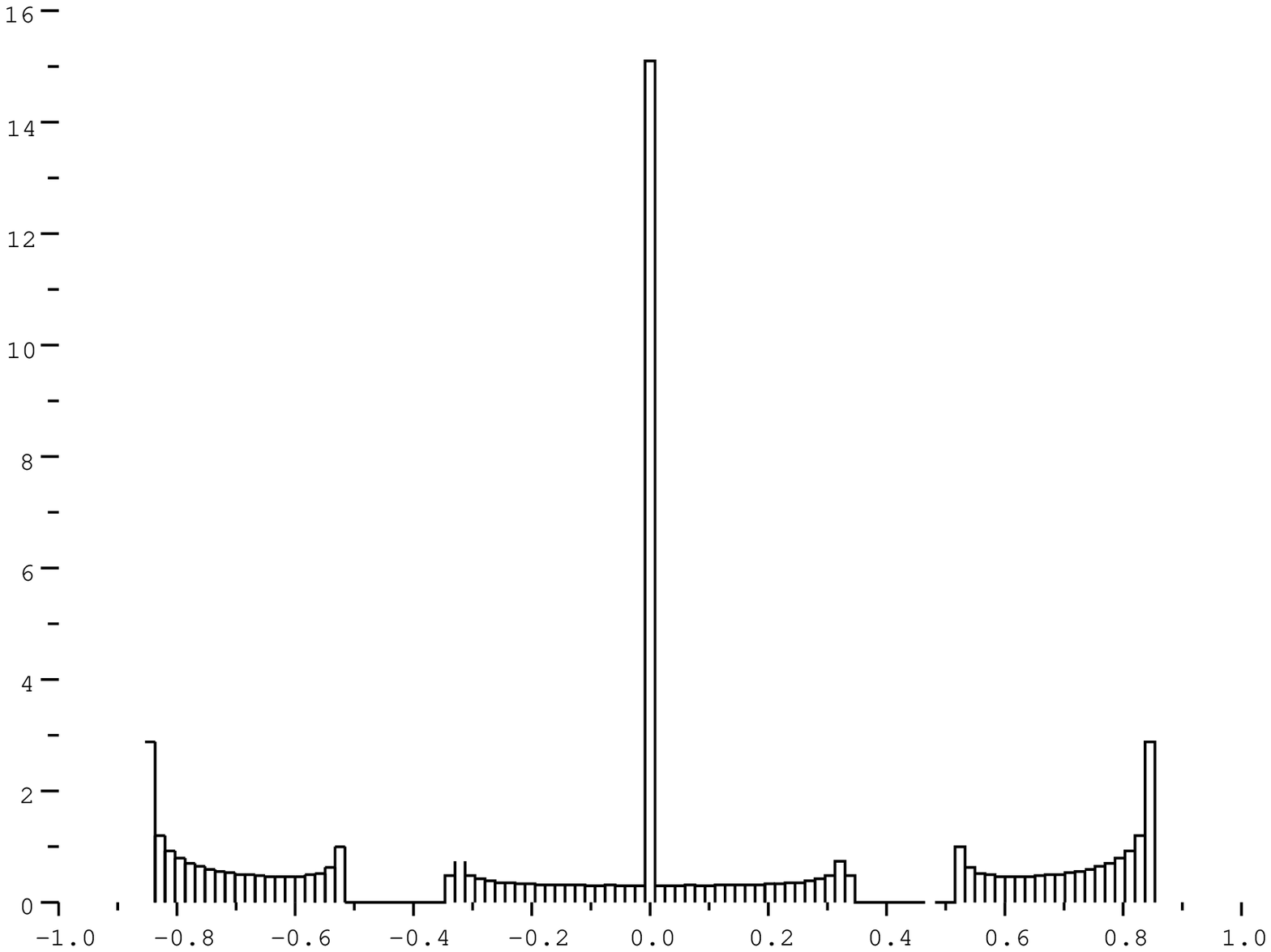}
 % graph_spec.eps: 0x0 pixel, 300dpi, 0.00x0.00 cm, bb=0 0 610 460
 \label{f:spectest}
\end{figure}

 To compute $\left(\mathcal{K}_n(f_\theta)\right)^{-1}$, we use the power series representation of $f_\theta$, and truncate
 this expression after the $15$ first coefficient.
This choice ensures that the simulation errors are neglectible with respect to the theoretical ones.

Figure \ref{densemp}
 gives the empirical distribution of 
$$\sqrt{m_n} \sqrt{\int_{\operatorname{Sp}(A)} \left(\frac{f'_\theta}{f_\theta}\right)^2}\left(\tilde{\theta}_n-\theta_0\right).$$

% We observe that given the size of the subgraphs chosen, the error is a little less concentrated than the asymptotic error (in red) which is a $\mathcal{N}(0,1)$.

\begin{figure}[htbp]
 \centering
\caption{Empirical distribution}
 \includegraphics[scale=0.5,keepaspectratio=true]{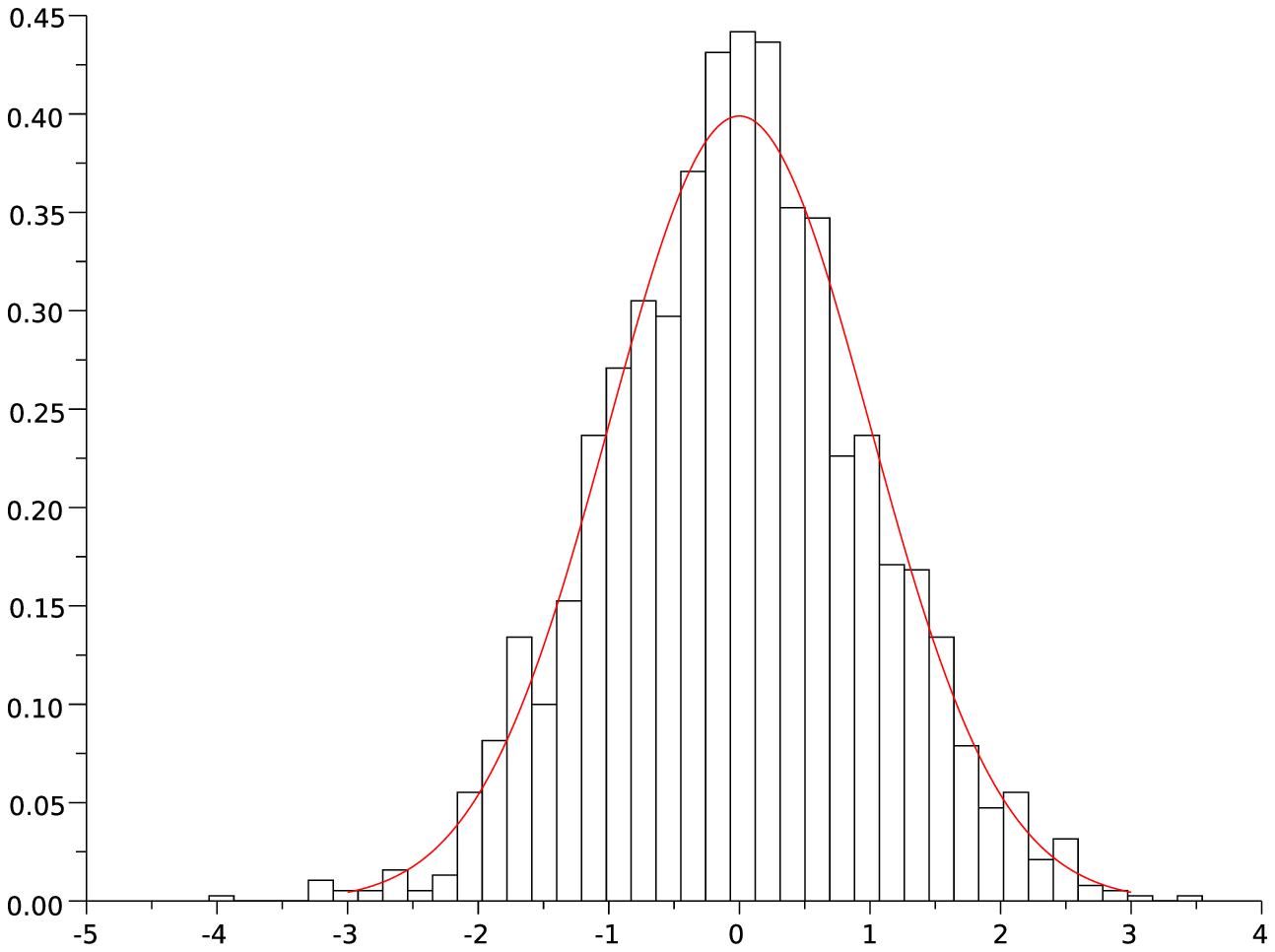}
%  \includegraphics[bb=0 0 610 500,scale=0.5,keepaspectratio=true]{dens_norm_3.eps}
 % histo.eps: 0x0 pixel, 300dpi, 0.00x0.00 cm, bb=0 0 610 460
 \label{densemp}
\end{figure}

 \newpage
\section{Appendix}\label{app}
\subsection{Szegö's Lemmas}\label{s:szego}

Szegö's Lemmas \cite{sego} are useful in time series analysis. Indeed, they provide good approximations for the likelihood. As explained in Section \ref{s:main_thm}, 
these approximations of the likelihood are easier to compute.

In this section, we generalize a weak version of the Szeg\"{o} Lemmas, for a general graph, 
under Assumption \ref{hyp_graph_1} (non expansion criterion for $G_n$), and Assumption \ref{hyp_graph_2} (existence of the spectral measure $\mu$).

For any matrix $(B_{ij})_{i,j \in G_n}$, we define the block norm 
 $$b_N(B) = \frac{1}{\delta_N} \sum_{i,j \in G_N} \left| B_{ij} \right|.$$

We can state the equivalent version of the first Szeg\"{o} lemma for time-series 

\begin{lem}{Asymptotic homomorphism}\label{lem_hom}

 Let $k,n$ be positive integers, and let $g_1, \cdots, g_k$ be analytic functions over $\left[-1 ,1 \right]$ having finite regularity factors (i.e. $\alpha(g_i) < +\infty, i = 1, \cdots, k$).
Then,
$$b_n\left(\mathcal{K}_{n}(g_1)\cdots \mathcal{K}_{n}(g_k)- \mathcal{K}_{n}(g_1\cdots g_k)\right) \leq \frac{k-1}{2}\alpha(g_1)\cdots \alpha(g_k) .$$
\end{lem}

\begin{cor} \label{c:det} For any $g \in \mathcal{F}_\rho$ (see the first page of Subsection \ref{ss:framework} for the definition), and under Assumptions \ref{hyp_graph_1} and \ref{hyp_graph_2},
$$\frac{1}{m_n}\log \det(\mathcal{K}_n(g)) \underset{n \rightarrow \infty}{\rightarrow} \int \log(g) \mathrm{d}\mu.$$ 
\end{cor}

\begin{proof} of Lemma \ref{lem_hom}
This proof follows again the one of \cite{AzDa}. 
We will prove the result by induction on $k$. 

First we deal with the case $k = 2$. 
Let $f$ and $g$ analytic functions over $\left[-1 ,1 \right]$ such that $\alpha(f) < + \infty$ and $\alpha(g) <+\infty$.
We write \begin{eqnarray*}& & b_n(\mathcal{K}_{n}(f) \mathcal{K}_{n}(g)- \mathcal{K}_{n}(f g)) 
 \\          &=& \frac{1}{\delta_n} \sum_{i,j \in G_n} \left| \sum_{k\in G_n}\left(\mathcal{K}_{n}(f)\right)_{ik} \left(\mathcal{K}_{n}(g)\right)_{kj} - \sum_{k\in G}\left(\mathcal{K}_{n}(f)\right)_{ik} \left(\mathcal{K}_{n}(g)\right)_{kj}     \right|
\\& =& \frac{1}{\delta_n} \sum_{i,j \in G_n} 
\sum_{k \in G \backslash G_n}\left|\mathcal{K}(f)_{ik} \right|\left|\mathcal{K}(g)_{kj} \right|.
         \end{eqnarray*}

Using $\mathcal{K}(g) = \sum_{h = 0}^\infty g_h W^h $, Fubini's theorem gives, since all the previous sequences are in $l^1(G)$, 
\begin{eqnarray*}
 & & b_n(\mathcal{K}_{n}(f) \mathcal{K}_{n}(g)- \mathcal{K}_{n}(f g)) 
\\ & \leq & \frac{1}{\delta_n} \sum_{i,j \in G_n} \sum_{k\in G \backslash G_n} \left| \left(\mathcal{K}_{n}(f)\right)_{ik} \left(\mathcal{K}_{n}(g)\right)_{kj} 
\right|
\\& \leq &  \left( \sup_{k \in G \backslash G_n} 
\sum_{i \in G_n} \left|\mathcal{K}(f)_{ik} \right| \right) \times
 \frac{1}{\delta_n}\sum_{k \in G \backslash G_n} \sum_{j \in G_n} \sum_{h = 0}^\infty \left| g_h \right| \left| (W^h)_{kj} \right|
\\ &\leq & \left( \sup_{k \in G } \sum_{i \in G} \left|\mathcal{K}(f)_{ik} \right| \right) \times 
  \sum_{h = 0}^\infty \left| g_h \right|  \frac{1}{\delta_n} \sum_{k \in G \backslash G_n} \sum_{j \in G_n}\left| (W^h)_{kj} \right|.
\end{eqnarray*}

Introducing 
 $$\Delta_h = \sup_{N \in \mathbb{N}} \frac{1}{\delta_N} \sum_{k \in G \backslash G_N} 
\sum_{j \in G_N} \left|\left( W^{h}\right)_{kj}\right|, $$
we get 
$$ b_n(\mathcal{K}_{n}(f) \mathcal{K}_{n}(g)- \mathcal{K}_{n}(f g)) \leq \sup_{k \in G } \sum_{i \in G} \left|\mathcal{K}(f)_{ik} \right|
  \sum_{h = 0}^\infty \left| g_h \right|  \Delta_h.$$

The coefficient $\Delta_h$ is a porosity factor. It measures the weight of the paths of length $h$ going from the interior of $G_n$ to outside.

Note that $\Delta_h \leq h+1$, so we get $$ \sum_{h = 0}^\infty \left| g_h \right|  \Delta_h \leq \alpha(g).$$

Now, we define another norm on $B_G$ :
$$ \left\| B\right\|_{\infty,in}  :=  \sup_{k \in G } \sum_{i \in G} \left| B_{ik} \right|,\left( B \in B_G\right).$$
We thus obtain
\begin{eqnarray*}
\left\| \mathcal{K}(f)\right\|_{\infty,in} & = & \sup_{k \in G } \sum_{i \in G} \left|\mathcal{K}(f)_{ik} \right| 
\\ & \leq & \sum_{h=0}^\infty \left| f_h \right| \left\| W^h\right\|_{\infty,in} 
\\ &\leq & \sum_{h=0}^\infty \left| f_h \right| \left\| W\right\|_{\infty,in}^h
\\ & \leq & \sum_{h=0}^\infty \left| f_h \right| := \left\| f \right\|_{1,pol}.
\end{eqnarray*}

Finally, we get 

$$b_n(\mathcal{K}_{G_n}(f) \mathcal{K}_{G_n}(g)- \mathcal{K}_{G_n}(f g)) \leq \left\| f \right\|_{1,pol} \alpha(g).$$

% Note that this inequality is tighter than the one of Lemma \ref{lem_hom}, and may be useful. 

To conclude the proof of the lemma, by symmetrization of the last inequality, and since $1 \leq (h+1)$, we have,

\begin{equation}
b_n\left(\mathcal{K}_{n}(f) \mathcal{K}_{n}(g)- \mathcal{K}_{n}(f g)\right) \leq \frac{1}{2}\alpha(f) \alpha(g).
\end{equation}

To perform the inductive step, we need the following inequalities \cite{matrixhandbook}: 
\begin{eqnarray*}
 \alpha(f g) &\leq& \alpha(f) \alpha(g),
\\ b_n(BC) &\leq &\left\|B\right\|_{\infty,in} b_n(C)  ,
\\ b_n(B+C) &\leq& b_n(B)+b_n(C) ,
\\ \left\|\mathcal{K}_{n}(f)\right\|_{\infty,in}& =& \left\|f\right\|_{1,pol} \leq \alpha(f).
\end{eqnarray*}

Let $k>1$, and assume that for all $j\leq k-1$, Lemma \ref{lem_hom} holds.
Under the previous assumptions, and the inductive hypothesis for $k-1$ we get, 
\begin{eqnarray*}
 b_n\left(\mathcal{K}_{n}(g_1)\times \right. &\cdots & \times \mathcal{K}_{n}(g_k)  -\left. \mathcal{K}_{n}(g_1\cdots g_k)\right) 
\\ &\leq& 
\left\| \mathcal{K}_{n}(g_1) \right\|_{\infty,in} b_n\left(\mathcal{K}_{n}(g_2) \cdots \mathcal{K}_{n}(g_k) - \mathcal{K}_{n}(g_2\cdots g_k) \right)
\\ & &+ b_n\left(\mathcal{K}_{n}(g_1)\mathcal{K}_{n}(g_2\cdots g_k)- \mathcal{K}_{n}(g_1\cdots g_k)\right)
\\ & \leq &  \alpha(g_1)\frac{k-2}{2}\alpha(g_2) \cdots \alpha(g_k) +  \frac{1}{2} \alpha(g_1)\alpha(g_2 \cdots g_k)
\\ & \leq & \frac{k-1}{2}\alpha(g_1) \cdots \alpha(g_k),
\end{eqnarray*}
which completes the induction step and proves the result.
\end{proof}

\begin{proof} of Corollary \ref{c:det}

Let $g\in \mathcal{F}_\rho$, and $k$ be a positive integer. Using Lemma \ref{lem_hom}, we have

\begin{equation} \label{e:mom} \operatorname{Tr}\left( \mathcal{K}_{n}(g)^k  - \mathcal{K}_{n}(g^k)  \right) \leq \frac{\delta_n}{m_n}b_n\left(\mathcal{K}_{n}(g)^k  - \mathcal{K}_{n}(g^k)\right). \end{equation}

Thus, we have, thanks to Assumption \ref{hyp_graph_1}
$$\frac{1}{m_n}\operatorname{Tr}\left( \mathcal{K}_{n}(g)^k  - \mathcal{K}_{n}(g^k)  \right) \underset{n\rightarrow +\infty}{\rightarrow} 0.$$

Denote $\mu_g^{[1]}$ the real measure whose $k^\text{th}$-moment is given by 
$$\int x^k \mathrm{d}\mu_g^{[1]} = \lim_n \frac{1}{m_n}\operatorname{Tr}\left( \mathcal{K}_{n}(g)^k \right),$$ 
and
$\mu_g^{[2]}$ the real measure whose $k^\text{th}$-moment is given by 
$$\int x^k \mathrm{d}\mu_g^{[2]}=\lim_n \frac{1}{m_n}\operatorname{Tr}\left( \mathcal{K}_{n}(g^k) \right).$$

Notice that both of these measures have support between $\inf g \geq e^{-\rho}>0$ and $\sup g\leq e^\rho <+\infty$, since $\alpha(\log(g))<\rho$ (see Section \ref{s:main_thm}). Therefore, the equality of the moments given by Equation \ref{e:mom} gives the equality of the measures $\mu_g^{[1]}$ and $\mu_g^{[2]}$.

So that, we get
\begin{equation}
\frac{1}{m_n}\log \left( \operatorname{det}\left( \mathcal{K}_n(g) \right) \right) -\frac{1}{m_n}\operatorname{Tr}\left(\mathcal{K}_n\left(\log(g)\right)\right)\underset{n\rightarrow +\infty}{\rightarrow} 0. 
\end{equation}

Assumption \ref{hyp_graph_2} completes the proof of the Corollary since it implies that
$$\frac{1}{m_n}\operatorname{Tr}\left(\mathcal{K}_n\left(\log(g)\right)\right)\underset{n\rightarrow +\infty}{\rightarrow} \int \log(g) \mathrm{d}\mu .$$

\end{proof}

The following lemma enables to replace $\mathcal{K}_n(g)$ by the unbiased version $\mathcal{Q}_n(g)$ (see Section \ref{s:main_thm} for the definition).
\begin{lem}\label{lem_unbiased}
Under Assumptions \ref{hyp_graph_1},\ref{hyp_graph_2}, \ref{a:graph_structure} and \ref{hyp_unbiased}, and if $f$ or $g$ is a polynomial having degree less than or equal to $P$, we have
 $$\left| \frac{1}{m_n}\operatorname{Tr}\left( \left(\mathcal{K}_{n}(f)\mathcal{K}_{n}(g)\right)^p - \left(\mathcal{K}_{n}(f)\mathcal{Q}_{n}(g)\right)^p \right)\right| \leq 2^p u_n\alpha(f)^p\alpha(g)^p.$$
\end{lem}

\begin{proof}
We define, for any $f$, $$f_{abs}(x) = \sum_k \left| f_k \right| x^k.$$

Actually, the proof is based of the following idea: as soon as $f$ or $g$ is a polynomial having 
degree less than or equal to $P$, we have to control only the number of paths of length less than or equal to $P$ (counted with their weights).

Let $p$ be a positive number.
Recall that $\mathcal{Q}_n(\frac{1}{g}) = B^{(n)} \odot \mathcal{K}_n(\frac{1}{g})$ (see Section \ref{s:main_thm}), we have,

{\begin{eqnarray*}
&\frac{1}{m_n}&\left| \operatorname{Tr}\left( \left(\mathcal{K}_{n}(f)\mathcal{K}_{n}(\frac{1}{g})\right)^p -\left(\mathcal{K}_{n}(f)\mathcal{Q}_{n}(\frac{1}{g})\right)^p  \right)\right| 
\\ &\leq &
\frac{1}{m_n}\left| \sum_{i \in G_n} \sum_{i_0 = i, i_1, \cdots, i_{2p} = i} 
\prod_{l = 0 \cdots p} B^{(n)}_{i_{2l}i_{2l+1}} \mathcal{K}_n(\frac{1}{g})_{i_{2l} i_{2l+1}} \mathcal{K}_n(f)_{i_{2l+1}i_{2l+2}} \right.
\\ & & - \left.\frac{1}{m_n}\sum_{i \in G_n} \sum_{i_0 = i, i_1, \cdots, i_{2p} = i} 
\prod_{l = 0 \cdots p} \mathcal{K}_n(\frac{1}{g})_{i_{2l} i_{2l+1}} \mathcal{K}_n(f)_{i_{2l+1}i_{2l+2}} \right|
\\ & \leq & \frac{1}{m_n} \sup_{i_1,i_2, \cdots, i_{2p+1}} \left|  \prod_{l = 0 \cdots p-1} B^{(n)}_{i_{2l+1}i_{2l+2}} -1 \right|       
\\ & & \times      \sum_{i \in G_n} \sum_{i_0 = i, i_1, \cdots, i_{2p} = i} 
\prod_{l = 0 \cdots p} \left| \mathcal{K}_n(\frac{1}{g})_{i_{2l} i_{2l+1}} \mathcal{K}_n(f)_{i_{2l+1}i_{2l+2}} \right|
\\ & \leq & \frac{1}{m_n} \sup_{i_1,i_2, \cdots, i_{2p+1}} \left|  \prod_{l = 0 \cdots p-1} B^{(n)}_{i_{2l+1}i_{2l+2}} -1 \right|    
\\ & & \times        \sum_{i \in G_n} \sum_{i_0 = i, i_1, \cdots, i_{2p} = i} 
\prod_{l = 0 \cdots p}  \mathcal{K}_n((\frac{1}{g})_{abs})_{i_{2l} i_{2l+1}} \mathcal{K}_n(f_{abs})_{i_{2l+1}i_{2l+2}} 
\\ & \leq & \sup_{i_1,i_2, \cdots, i_{2p+1}} \left|  \prod_{l = 0 \cdots p-1} B^{(n)}_{i_{2l+1}i_{2l+2}} -1 \right|      
\left\|  \left(  K_{G_n}(f_{abs})K_{G_n}((\frac{1}{g})_{abs})\right)^p \right\|_{2,in} 
\\ & \leq & \sup_{i_1,i_2, \cdots, i_{2p+1}} \left|  \prod_{l = 0 \cdots p-1} B^{(n)}_{i_{2l+1}i_{2l+2}} -1 \right|  
\alpha(f)^p \alpha(\frac{1}{g})^p.
\end{eqnarray*}

Using Assumption \ref{hyp_unbiased}, we get,
\begin{eqnarray*}
 &\frac{1}{m_n}&\left| \operatorname{Tr}\left( \left(\mathcal{K}_{n}(f)\mathcal{K}_{n}(\frac{1}{g})\right)^p -\left(\mathcal{K}_{n}(f)\mathcal{Q}_{n}(\frac{1}{g})\right)^p  \right)\right|
\\ & \leq &  \left|  (1+u_n)^p-1 \right|  
\alpha(f)^p \alpha(\frac{1}{g})^p
\\ & \leq &  \left|  \left(1+u_n-1\right)\left((1+u_n)^{p-1}+(1+u_n)^{p-2} +\cdots +1\right) \right|  
\alpha(f)^p \alpha(\frac{1}{g})^p
\\ & \leq &  \left|  u_n\left(2^p-1\right) \right|  
\alpha(f)^p \alpha(\frac{1}{g})^p
\\ & \leq & u_n 2^p\alpha(f)^p\alpha(\frac{1}{g})^p  .
\end{eqnarray*}}
This ends the proof of the Lemma.
\end{proof}
Finally, the following lemma explains the choice of $B^{(n)}$. The unbiased quadratic form $\mathcal{Q}_n$ is no more than a correction of the error
between $\mathcal{K}_n(f)\mathcal{K}_n(g)$ and $\mathcal{K}_n(fg)$.

\begin{lem}[Exact correction]
\label{l:correction}
Let $f,g \in \mathcal{F}_\rho$, and assume that either $f$ or $g$ is a polynomial of degree less than or equal to $P$ (see Section \ref{s:main_thm}). Then, the unbiased quadratic form $\mathcal{Q}_n(f_\theta)$ verify
$$\operatorname{Tr}\left(\mathcal{K}_n(f)\mathcal{Q}_n(g) \right) = \operatorname{Tr}\left(\mathcal{K}_n(fg) \right). $$
\end{lem}
\begin{proof} of Lemma \ref{l:correction}

First, notice that 
$$\operatorname{Tr}\left(\mathcal{K}_n(f)\mathcal{Q}_n(g) \right) = \sum_{i,j \in G_n}  \mathcal{K}_n(f)_{ij} \mathcal{K}_n(g)_{ij} B^{(n)}_{ij}.$$
Since this expression is symmetric on $f,g$, we can now consider the case where $f$ is a polynomial of degree less than or equal to $P$.

% Write, for $i,j,k,l ,t(i,j) = t(k,l)$, 
% $$K(f)_{ij} = \sum_{p \leq P} f_p \left(W^{p}\right)_{ij} = \sum_{p \leq P} f_p \left(W^{p}\right)_{kl} = K(f)_{kl}.$$
Actually, since $f$ is a polynomial, $\mathcal{K}_n(f)_{ij} = 0$ as soon as $d(i,j) > P$ ($i,j \in G$).
Then, if $i,j,k,l \in G$ are such that $\mu_{ij}= \mu_{kl}$, we have
$$\mathcal{K}_n(f)_{ij} \mathcal{K}_n(g)_{ij} = \mathcal{K}_n(f)_{kl} \mathcal{K}_n(g)_{kl}.$$ 
So that, we may here denote, for convenience, $K(f)_{\mu_{ij}}$.

Using Assumption \ref{a:graph_structure}, this leads to 
\begin{eqnarray*}
 \operatorname{Tr}\left(\mathcal{K}_n(f)\mathcal{Q}_n(g) \right) & = &\sum_{i,j \in G_n}  \mathcal{K}_n(f)_{ij} \mathcal{K}_n(g)_{ij} B^{(n)}_{ij}
\\ & =& \sum_{v \in V_P} \sum_{ i,j \in G_n \atop  \mu_{ij}=v,  d_{\mathbf{G}}(i,j) \leq P}  \mathcal{K}_n(f)_{v} \mathcal{K}_n(g)_{v} B^{(n)}_v
\\ & = & \sum_{v \in V_P} \mathcal{K}_n(f)_{v} \mathcal{K}_n(g)_{v}  \text{Card}\left\{(i,j) \in G_n\times G_n, \mu_{ij}= v \right\}  
\\ & & \times 
\frac{\text{Card}\left\{(i,j) \in G_n\times G, \mu_{ij}= v \right\}       }
{ \text{Card}\left\{(i,j) \in G_n\times G_n, \mu_{ij}= v \right\}  },
\\ & =&\sum_{v \in V_P} \sum_{(i,j) \in G_n \times G, \atop \mu_{ij}=v, d_{\mathbf{G}}(i,j) \leq P}  \mathcal{K}_n(f)_{v} \mathcal{K}_n(g)_{v} B^{(n)}_v
\\ & = &\sum_{(i,j) \in G_n \times G}  \mathcal{K}_n(f)_{ij} \mathcal{K}_n(g)_{ij} B^{(n)}_{ij}
\\ &= &\operatorname{Tr}\left(\mathcal{K}_n(fg) \right).
\end{eqnarray*}
That ends the proof of Lemma \ref{l:correction}.
%\input{About_the_assumptions_2010_10_26.tex}
% \newpage
% \input{Example_traffic_2010_10_19.tex}
\end{proof}

\subsection{Proofs of the lemmas of Theorem \ref{thm_conv}}
\label{ss:proof_lemmas_main_thm}
Recall that the theorem relies on two lemmas. Lemma \ref{lem_conv} states a condition on deterministic sequences to provide the convergence of the maximizer 
of these sequences.

\begin{proof} of Lemma \ref{lem_conv}
Recall that $f_{\theta_0}$ denotes the true spectral density.
Let $(\ell_n)_{n \in \mathbb{N}}$ be a deterministic sequence of continuous functions such that 
\begin{equation}
  \forall \theta \in \Theta, \ell_n(\theta_0) - \ell_n(\theta) \underset{n \rightarrow \infty}{\rightarrow}  \frac{1}{2}\int\left( -\log(\frac{f_{\theta_0}}{f_\theta}) - 1 + \frac{f_{\theta_0}}{f_\theta}\right) \mathrm{d} \mu. 
\end{equation}
 uniformly as $n$ tends to infinity.
Denotes moreover $\theta_n = \arg \max_\theta \ell_n(\theta)$.
We aim at proving that
$$\theta_n \underset{n \rightarrow \infty}{\rightarrow} \theta_0 .$$

Using the compactness of $\Theta$, let $\theta_\infty$ be an accumulation point of the sequence $(\theta_n)_{n \in \mathbb{N}}$, and $(\theta_{n_k})_{k \in \mathbb{N}}$ be a subsequence converging to $\theta_\infty$.
As the function $$\theta \mapsto \frac{1}{2}\int \left(-\log(\frac{f_{\theta_0}}{f_\theta}) - 1 + \frac{f_{\theta_0}}{f_\theta}\right) \mathrm{d} \mu $$ 
is continuous on $\Theta$,
and the convergence of $(\ell_n(\theta_0) - \ell_n(\theta))_{n \in \mathbb{N}}$ is uniform in $\theta$, we have
\begin{equation} \ell_{n_k}(\theta_0) - \ell_{n_k}( \theta_{n_k}) \xrightarrow{k \rightarrow \infty}  \frac{1}{2}\int -\log(\frac{f_{\theta_0}}{f_{\theta_\infty}}) - 1 + \frac{f_{\theta_0}}{f_{\theta_\infty}}
 \mathrm{d} \mu.
\end{equation} 
But we can notice that, thanks to the definition of $\theta_n$, $ \ell_{n_k}(\theta_0) - \ell_{n_k}( \theta_{n_k}) \leq 0 $
So, since the function $x \mapsto -\log(x)+x-1 $ is non negative and vanishes if, and only if, $x=1$, we get that $f_{\theta_0} = f_{\theta_{\infty}}$. 
By injectivity of the function $\theta \rightarrow f_\theta $, we get $\theta_\infty=\theta_0$, for any accumulation point
 $\theta_\infty$ of the sequence $(\theta_n)_{n \in \mathbb{N}}$, which ends the proof of this first lemma.
\end{proof}

Lemma \ref{lem_ik} provides the uniform convergence of the contrasts of maximum likelihood and approximated maximum likelihood to the Kullback information. The proof may be cut into several lemmas.
% 
% \begin{lem}
% Under Assumptions \ref{hyp convergence}, the asymptotic Kullback information exists and verify 
% $$IK(f,g) = \frac{1}{2}\int -\log(\frac{f}{g}) - 1 + \frac{f}{g} \mathrm{d} \mu  $$ 
% 
% 
% 
% Furthermore, if we denote $l_n(\theta,X) = \frac{1}{m_n}L_n(\theta,X)$, we have $P_{f_{\theta_0}}$-almost surely,
%  $$ l_n(\theta_0,X) - l_n( \theta,X) \rightarrow IK(f_{\theta_0},f_\theta)= \frac{1}{2}\int -\log(\frac{f_{\theta_0}}{f_\theta}) - 1 + \frac{f_{\theta_0}}{f_\theta} \mathrm{d} \mu$$
% uniformly as $n$ tends to infinity.
% This property is still true if we change $L_n$ into $\bar{L}_n$ or $\tilde{L}_n$.
% \end{lem}

\begin{proof} of Lemma \ref{lem_ik}

% Furthermore, if we set $l_n(\theta,X_n) = \frac{1}{m_n}L_n(\theta,X_n)$, we have that $P_{f_{\theta_0}}$-a.s.,
%  $$ l_n(\theta_0,X_n) - l_n( \theta,X_n) \underset{n \rightarrow \infty}{\rightarrow} \mathbb{IK}(f_{\theta_0},f_\theta) %= \frac{1}{2}\int -\log(\frac{f_{\theta_0}}{f_\theta}) - 1 + \frac{f_{\theta_0}}{f_\theta} \mathrm{d} \mu
% $$
% uniformly in $\theta \in \Theta$.
% 
% This property also holds for $\bar{l}_n := \frac{1}{m_n}\bar{L}_n$ and $\tilde{l}_n:=\tilde{L}_n$ 
% 
% Furthermore, for $P>0$, and for both the $AR_P$ or the $MA_P$ case (see above), this also holds for $l_n^{(u)} := \frac{1}{m_n}L^{(u)}_n$.
% \end{lem}

First, notice that by construction, we have, for any $\theta \in \Theta$,
\begin{equation}
\mathbb{IK}(f_{\theta_0},f_\theta) = \lim_n \mathbb{E} \left[\frac{1}{m_n} \left( L_n(f_{\theta_0},X_{n}) - L_n(f_\theta,X_{n})\right) \right] , 
\end{equation}
when it exists.
Then, we can compute 
\begin{eqnarray*}
 l_n(f_{\theta_0},X_n) - l_n(f_\theta,X_n) &=& -\frac{1}{2m_n} \left(\log \det(\mathcal{K}_{n}(f_{\theta_0})) - \log \det(\mathcal{K}_{n}(f_\theta))\right)
\\ & &  -\frac{1}{2m_n}\left(X_n^T\mathcal{K}_{n}(f_{\theta_0})^{-1}X_n-X_n^T\mathcal{K}_{n}(f_\theta)^{-1}X_n\right)
\end{eqnarray*}

Corollary \ref{c:det} of Lemma \ref{lem_hom} provides the following convergence
\begin{equation}
\frac{1}{m_n} \left(\log \det(\mathcal{K}_{n}(f_{\theta_0})) - \log \det(\mathcal{K}_{n}(f_\theta))\right)\underset{n \rightarrow \infty}{\rightarrow}\int \log\left(\frac{f_{\theta_0}}{f_\theta}\right)\mathrm{d}\mu. 
\end{equation}

To prove the existence of $\mathbb{IK}(f_{\theta_0},f_\theta)$, it only remains to prove the $\mathbb{P}_{f_{\theta_0}}$-a.s. convergence of $\frac{1}{m_n}X_n^T\mathcal{K}_{n}(f_\theta)^{-1}X_n$ to $\int\frac{f_{\theta_0}}{f_\theta} \mathrm{d} \mu $ as $n$ goes to infinity.

This is ensured by the following Lemma.
%  which provides also the convergence to the Kullback information for $\bar{l}_n$, $\tilde{l}_n$, and $l_n^{(u)}$ in the $AR_P$ or $MA_P$ cases (see \ref{s:main_thm}).

%\begin{lem}
%Under homogeneity assumption, we get that, in the sense of the weak convergence, 
%$$\mu_n \rightarrow \mu$$
%\end{lem}
% Let $f,g \in \mathcal{F}_\rho, \rho >0$, we will now provide several lemmas. Together, they will ensure, the desired convergences.
% 
% 
% \begin{lem}[Weak convergence of the measure]\label{lem_conv_measure}
%  For any $\phi$ continuous bounded function $\phi$ from $\text{Im}(g)$ to $\mathbb{R}$, we have
% $$\frac{1}{m_n}\left(Tr(\phi(K_n(g))) - K_n(\phi(g)) \right) \xrightarrow{ n \rightarrow \infty} 0 .$$
% \end{lem}
% 
% 
% \begin{cor}[Determinant lemma]\label{lem_det}
% It holds that
% $$\left|\frac{1}{m_n} \log( \det(K_{G_n}(g))) - \int \log(g)\mathrm{d}\mu \right| \underset{n \rightarrow \infty}{\rightarrow} 0. $$
% \end{cor}

\begin{lem}[Convergence lemma]\label{concentration}
For respectively $\Lambda = \mathcal{K}_{n}(\frac{1}{f_\theta})$, $\Lambda = (\mathcal{K}_{n}(f_\theta))^{-1}$ or $\Lambda = \mathcal{Q}_n(\frac{1}{f_\theta})$, we have,
$$\frac{1}{m_n} X_n^T\Lambda X_n \underset{n \rightarrow \infty}{\rightarrow} \int \frac{f_{\theta_0}}{f_\theta} \mathrm{d}\mu, \mathbb{P}_{f_{\theta_0}}-\text{a.s.}.$$
\end{lem}

Lemma \ref{concentration} combined with Corollary \ref{c:det} ensures the $\mathbb{P}_{f_{\theta_0}}-\text{a.s.}$ convergence of $\tilde{l}_n(f_{\theta_0})-\tilde{l}_n(f_\theta)$, $ \bar{l}_n(f_{\theta_0})-\bar{l}_n(f_\theta)$ to $\mathbb{IK}(f_{\theta_0},f_\theta)$.
It provides also the $\mathbb{P}_{f_{\theta_0}}-\text{a.s.}$ convergence of $l^{(u)}_n(f_{\theta_0})-l^{(u)}_n(f_\theta)$ to $\mathbb{IK}(f_{\theta_0},f_\theta)$ in the $AR_P$ or $MA_P$ cases (see Section \ref{s:main_thm}).  
To complete the assertion of Lemma \ref{lem_ik}, it only remains to show the uniform convergences on $\Theta$ of the last quantities. This will be done using an equicontinuity argument given by the following Lemma.

% is a consequence of the compactness of $\Theta$ and of the following Lemma.
% Combined with 
% lemma 
% \ref{lem_det}, it leads to the convergence of $l_n(f)-l_n(g)$ to $IK(f,g)$. Now, only the uniform convergence of this expressions remains to be shown. It is sufficient to provide 
% the equicontinuity of the sequence of function. 

\begin{lem}[Equicontinuity lemma]\label{equicont}
For all $n \geq 0$, the sequences of functions 
$$\left(l_n(f_{\theta_0},X_n) - l_n(f_\theta,X_n)\right)_{n \in \mathbb{N}}$$
 is an $\mathbb{P}_{f_{\theta_0}}$-a.s. equicontinuous sequence on $\left( \left\{f_\theta, \theta \in \Theta \right\}, \left\|.\right\|_\infty\right)$.%, for the topology of the uniform convergence.
%$\Theta$, $\mathbb{P}_{f_{\theta_0}}$-a.s..
 This property also holds for
 $\bar{l_n}$,$\tilde{l}_n$.
Furthermore, the sequence $\left(l^{(u)}_n(f_{\theta_0},X_n -l_n^{(u)}(f_\theta,X_n)\right)_{n \in \mathbb{N}}$ is also   $\mathbb{P}_{f_{\theta_0}}$-a.s. equicontinuous, on $\left(\left\{f_\theta, \theta \in \Theta \right\}, \left\|.\right\|_{1,pol}\right)$.
\end{lem}

We can now end the proof of Lemma \ref{lem_ik}:

First, notice that the space $\left\{ f_\theta, \theta \in \Theta\right\}$ is compact for the topology of the uniform convergence. This also holds for  $\left(\left\{f_\theta, \theta \in \Theta \right\}, \left\|.\right\|_{1,pol}\right)$.
So, there exists a dense sequence $(f_{\theta_p})_{p \in \mathbb{N}}$. Then, using Lemma \ref{lem_hom} and Corollary \ref{c:det}, the sequence
$\left(l_n(f_{\theta_0},X_n) - l_n(f_{\theta_p},X_n)\right)_{n \in \mathbb{N}}$ converges  $\mathbb{P}_{f_{\theta_0}}$-a.s. to $\mathbb{IK}(f_{\theta_0},f_{\theta_p})$.
 
If a sequence of functions is equicontinuous and converges pointwise on a dense subset of its domain, and if its co-domain is a complete space, then the sequence converges pointwise on all the domain \cite{rudin}.

Using this well known property, we obtain, $\mathbb{P}_{f_{\theta_0}}$-a.s., the pointwise convergence of $$\left(l_n(f_{\theta_0},X_n) - l_n(f_{\theta},X_n)\right)_{n \in \mathbb{N}}$$ to $\mathbb{IK}(f_{\theta_0},f_{\theta})$, for any $\theta \in \Theta$.

Furthermore, if a sequence of functions is equicontinuous and converges pointwise on its domain, then this convergence is uniform on any compact subspace of the domain \cite{rudin}.

Thus, we get, $\mathbb{P}_{f_{\theta_0}}$-a.s., the uniform convergence on $\Theta$ of the sequence $$\left(l_n(f_{\theta_0},X_n) - l_n(f_{\theta},X_n)\right)_{n \in \mathbb{N}}$$ to $\mathbb{IK}(f_{\theta_0},f_{\theta})$.

Using the same kind of arguments, this uniform convergence also holds for $\bar{l_n}$,$\tilde{l}_n$ and $l^{(u)}_n$. This concludes the proof of Lemma \ref{lem_ik}.
%COMMENTAIRE UTILE : Si une suite (fn) de fonctions est équicontinue et converge
% simplement sur un sous-ensemble dense de l'espace de départ, et si l'espace d'arrivée est complet, 
%alors la suite converge simplement sur l'espace de départ tout entier*

% By compacity and equicontinuity of $\left\{g_\theta, \theta \in \Theta \right\}$, this convergence is furthermore uniform on $\Theta$.
% This ends the proof of the main lemma.
%COMMENTAIRE UTILE : Si une suite (fn) de fonctions est équicontinue et converge simplement alors cette convergence est uniforme sur tout compact.
% Plus généralement, si K est un espace compact et si A
% est un ensemble équicontinu de fonctions de K dans F alors sur A, la topologie de la convergence simple et celle de la convergence uniforme coïncident.

\end{proof}

\subsection{Proof of the technical lemmas}
\label{s:technical_lemmas}

% \begin{lem}[Weak convergence of the measure above]
%  For any $\phi$ a continuous bounded function, we get that
% $$\frac{1}{m_n}\left(Tr(\phi(K_n(g))) - K_n(\phi(g)) \right) \rightarrow 0 $$
% \end{lem}

% \begin{proof} of Lemma \ref{lem_conv_measure}
% Since the functions have compact support, it is sufficient to prove the result for $\phi(x) = x^k, k >0$. Using $\alpha(g)\leq \rho < +\infty$ and Lemma \ref{lem_hom}, we get 
% $$\frac{1}{m_n} Tr((K_n(g))^k - K_n(g^k)) \leq \frac{\delta_n}{m_n} k \alpha(g)^k \xrightarrow{k \rightarrow \infty} 0 .$$
% \end{proof}

% \begin{lem}[Concentration lemma]
%  We get, almost surely, the following convergence, for respectively $\Gamma = K_n(\frac{1}{g)}$, $\Gamma = (K_n(g))^{-1}$ and $\Gamma$ 
% the unbiased operator.
% $$X^T\Gamma X \rightarrow \int \frac{f}{g} \mathrm{d}\mu$$
% \end{lem}

\begin{proof} of Lemma \ref{concentration}
% 
% \begin{lem}[Concentration lemma]\label{concentration}
% For respectively $\Lambda = \mathcal{K}_{n}(\frac{1}{f_\theta})$, $\Lambda = (\mathcal{K}_{n}(f_\theta))^{-1}$ or $\Lambda = \mathcal{Q}_n(\frac{1}{f_\theta})$, we have,
% $$\frac{1}{m_n} X_n^T\Lambda X_n \underset{n \rightarrow \infty}{\rightarrow} \int \frac{f_{\theta_0}}{f_\theta} \mathrm{d}\mu, \mathbb{P}_{f_{\theta_0}}-\text{a.s.}.$$
% \end{lem}

Let $\theta \in \Theta$.
First, consider the case $\Lambda_n = \mathcal{K}_n\left(\frac{1}{f_\theta}\right)$.
We aim at proving that  
$$\frac{1}{m_n} X_n^T\Lambda_n X_n \underset{n \rightarrow \infty}{\rightarrow} \int \frac{f_{\theta_0}}{f_\theta} \mathrm{d}\mu, \mathbb{P}_{f_{\theta_0}}-\text{a.s.}.$$

To do that, we make use of classical tools of large deviation (see \cite{dembo}).
We compute the Laplace transform of $X_n^T\Lambda_n X_n$ :
\begin{eqnarray*}
&\mathbb{E}_{\mathbb{P}_{f_{\theta_0}}}& \left[ e^{\lambda X_{n}^T \mathcal{K}_n(\frac{1}{f_\theta}) X_{n}}\right] 
\\ & = & \frac{1}{(\sqrt{2\pi})^{m_n}\sqrt{\det(\mathcal{K}_n({f_{\theta_0}))}}}\int e^{\frac{1}{2} X_{n}^T \left( \left(\mathcal{K}_n(f_{\theta_0})\right)^{-1} - 2\lambda \mathcal{K}_n(\frac{1}{f_\theta}) \right)X_{n}}
\\ & = & \frac{1}{\sqrt{\det(\mathcal{K}_n(f_{\theta_0})})} 
\sqrt{\det \left( \left[ \left(\mathcal{K}_n(f_{\theta_0})\right)^{-1} - 2\lambda \mathcal{K}_n(\frac{1}{f_\theta}) \right]^{-1}\right)}
\\  & = & \frac{1}{\sqrt{\det\left( I_{G_n} -2 \lambda  \mathcal{K}_n(f_{\theta_0})^{\frac{1}{2}}\mathcal{K}_n(\frac{1}{f_\theta})\mathcal{K}_n(f_{\theta_0})^{\frac{1}{2}}    \right)}}.
\end{eqnarray*}
These last equalities hold as soon as $I_{G_n} -2 \lambda  \mathcal{K}_n(f_{\theta_0})^{\frac{1}{2}}\mathcal{K}_n(\frac{1}{f_\theta})\mathcal{K}_n(f_{\theta_0})^{\frac{1}{2}} $ is positive. This is true whenever $\lambda \leq 0$  or small enough.

Now, for $\lambda\leq 0$, define
\begin{equation*}
 \phi_n(\lambda) := \frac{1}{m_n} \log \left( \right. \mathbb{E}_{\mathbb{P}_{f_{\theta_0}} } \left. \left[ e^{\lambda X_{n}^T \mathcal{K}_n(\frac{1}{f_\theta}) X_{n}}\right] \right) ,
\end{equation*}
This function verifies
$$\phi_n(\lambda) =  -\frac{1}{2m_n} \log \det\left( I_{G_n} -2 \lambda \mathcal{K}_n(f_{\theta_0})^{\frac{1}{2}}\mathcal{K}_n(\frac{1}{f_\theta})\mathcal{K}_n(f_{\theta_0})^{\frac{1}{2}}       \right).$$

Define also 

$$ \phi(\lambda) = \lim_n \phi_n(\lambda),$$

We get, using Corollary \ref{c:det},
\begin{equation*}
 \phi(\lambda) =  -\frac{1}{2} \int \log \left( 1 -2 \lambda \frac{f_{\theta_0}}{f_\theta}       \right) .
\end{equation*}
 
We can also compute 
$$\phi''(\lambda) = \int \frac{2(\frac{f_{\theta_0}}{f_\theta})^2}{(1-2\lambda\frac{f_{\theta_0}}{f_\theta})^2} \mathrm{d}\mu>0.$$

As very usual, we define the convex conjugate of $\phi$ by 
$$\phi^*(t):= \sup_{\lambda \in \mathbb{R}^-} \left[ \lambda t -\phi(\lambda)\right] , t \in \mathbb{R} .$$
 
As soon as $\phi$ is strictly convex, $\phi^*(t)>\phi(0)=0$, for any $t \neq \phi'(0) = \int \frac{f}{g} \mathrm{d}\mu$.

We can now write, for $\lambda \leq 0$,
\begin{align*}
\frac{1}{m_n} \log(\mathbb{P}(\frac{1}{m_n}X_n^T \Lambda_n X_n \geq t)) & = \frac{1}{m_n} \log(\mathbb{P}(e^{ \lambda X_n^T \Lambda_n X_n} 
\geq e^{ m_n\lambda t}))
\\ & \leq \frac{1}{m_n}\log \left( e^{-m_n \lambda t}  \right)+  \frac{1}{m_n} \log \left(\mathbb{E}[e^{ \lambda X_n^T \Lambda_n X_n}] \right)
\\ & \leq   -\lambda t + \phi_n(\lambda) .
\end{align*}
%ici 4/12/2010
Then we get, $\forall t > \int \frac{f}{g} \mathrm{d}\mu$,
$$ \limsup_n \left(\frac{1}{m_n} \log(\mathbb{P}(\frac{1}{m_n}X_n^T \Lambda_n X_n \geq t)) \right) \leq - \lambda t + \phi(\lambda)$$

So that, taking the infimum on $\lambda$, we get
$$\limsup_n \left(\frac{1}{m_n} \log(\mathbb{P}(\frac{1}{m_n}X_n^T \Lambda_n X_n \geq t)) \right) \leq  - \phi^*(t) < 0  $$

We can obtain the same bound for $t < \int \frac{f}{g}\mathrm{d}\mu$. By Borel-Cantelli theorem, we get the $\mathbb{P}_{f_{\theta_0}}$-almost sure 
convergence of $\frac{1}{m_n}X_n^T \Lambda_n X_n$ to $\int \frac{f}{g} \mathrm{d}\mu$. To prove the same convergence with $\Lambda_n = (\mathcal{K}_n(f_\theta))^{-1}$, we have to show that the difference between the spectral empirical measure of $ \mathcal{K}_n(f_{\theta_0})^{\frac{1}{2}}\mathcal{K}_n(\frac{1}{f_\theta})\mathcal{K}_n(f_{\theta_0})^{\frac{1}{2}}    $
and $ \mathcal{K}_n(f_{\theta_0})^{\frac{1}{2}}\mathcal{K}_n(f_\theta)^{-1}\mathcal{K}_n(f_{\theta_0})^{\frac{1}{2}}    $ converges weakly to zero. It is sufficient to control the convergence of every moment, because these two last measures both have compact support.

For this, we make use of the Schatten norms.
For any $A,B $ matrices of $M_{m_n}(\mathbb{R})$, we define 
$$\left\| A \right\|_{Sch,p} = \left(\sum s_k(A)^p\right)^{\frac{1}{p}}, $$
where $s_k(A)$ are the singular values of $A$.

Note that
$$\left| Tr(AB) \right| \leq \left\|AB \right\|_{Sch,1} \leq \left\| A \right\|_{Sch,1} \left\| B \right\|_{Sch,\infty}.$$

Recall that since $f_\theta \in \mathcal{F}_\rho$, we have $e^{-\rho} \leq f_\theta \leq e^\rho$.
Hence, for any $p \geq 1$,
\begin{align*}
 \frac{1}{m_n}  \left|\operatorname{Tr}\Bigl{(}\mathcal{K}_n^p{(\frac{1}{f_\theta})}\right. &\mathcal{K}_n^p(f_{\theta_0})  -\left. \mathcal{K}_n^{-p}(f_\theta)\mathcal{K}_n^p(f_{\theta_0}) \Bigl{)} \right|
\\ & \leq \frac{1}{m_n} \left\| \mathcal{K}_n{(f_\theta)}^{-p}\mathcal{K}_n^p(f_{\theta_0}) \right\|_{Sch,\infty} 
\left\| \left(\mathcal{K}_n^p{(\frac{1}{\theta})} \mathcal{K}_n^p(f_\theta)- I_{G_n}\right)\right\|_{Sch,1} 
\\ & \leq  \frac{\delta_n}{m_n} \frac{e^{2\rho p}}{e^{-2 \rho p}}\alpha(f_\theta)^{2p}\alpha(\frac{1}{f_\theta})^{2p} \underset{n \rightarrow \infty}{ \rightarrow} 0.
\end{align*}

To obtain the same bound with $\Lambda_n = \mathcal{Q}_{n}(\frac{1}{f_\theta})$, we have to prove that the difference between the spectral empirical measures of  $ \mathcal{K}_n(f_{\theta_0})^{\frac{1}{2}}\mathcal{K}_n(\frac{1}{f_\theta})\mathcal{K}_n(f_{\theta_0})^{\frac{1}{2}}    $
and $  \mathcal{K}_n(f_{\theta_0})^{\frac{1}{2}}\mathcal{Q}_n(\frac{1}{f_\theta})\mathcal{K}_n(f_{\theta_0})^{\frac{1}{2}}    $ converge weakly to zero. This last assertion is a direct consequence of Lemma \ref{lem_unbiased}. 
So, we get $$\frac{1}{m_n}X_n^T\Lambda_n X_n \rightarrow \int \frac{f_{\theta_0}}{f_\theta}, \mathbb{P}_{f_{\theta_0}}-\text{ a.s.}$$
 \end{proof}

% 
% \begin{lem}[Equicontinuity lemma]
% The functions $l_n(f) - l_n(g)$ are equicontinuous over $g_\theta, \theta \in \Theta$. This property remains if we substitute $l_n$ by $\bar{l_n}$,$\tilde{l}_n$ or $l^{(u)}_n$ 
% \end{lem} 

\begin{proof} of Lemma \ref{equicont}
% \begin{lem}[Equicontinuity lemma]\label{equicont}
% For all $n \geq 0$, the functions $\left(l_n(f_{\theta_0},X_n) - l_n(f_\theta,X_n)\right)_{n \in \mathbb{N}}$ is an $\mathbb{P}_{f_{\theta_0}}$-a.s. equicontinuous sequence on $\left\{f_\theta, \theta \in \Theta \right\}$.
% %$\Theta$, $\mathbb{P}_{f_{\theta_0}}$-a.s..
%  This property also holds for
%  $\bar{l_n}$,$\tilde{l}_n$ and $l^{(u)}_n$ 
% \end{lem}

Recall that we aim at proving that,  $\mathbb{P}_{f_{\theta_0}}$-a.s., the sequence of functions
 $$\left(l_n(f_{\theta_0},X_n) - l_n(f_\theta,X_n)\right)_{n \in \mathbb{N}}$$
 is  equicontinuous on $\left\{f_\theta, \theta \in \Theta \right\}$, and that this property also holds for $\bar{l_n}$,$\tilde{l}_n$ and $l^{(u)}_n$.

First, we will prove the equicontinuity of the sequence
$$\left( \frac{1}{m_n} \log \det (\mathcal{K}_n(f_\theta))\right)_{n \in \mathbb{N}}.$$
Let $\theta, \theta' \in \Theta$.

Denote $\lambda_i$ the eigenvalues of $\mathcal{K}_n(f_{\theta'})^{-1}\left( \mathcal{K}_n(f_{\theta'})- \mathcal{K}_n(f_\theta)\right)$.
Since $f_\theta \in \mathcal{F}_\rho$, we have $e^{-\rho} \leq f_\theta \leq e^\rho$.

Notice that we have
\begin{align*}
\sup_{i = 1, \cdots, n}\left| \lambda_i\right| &=  \left\| \mathcal{K}_n(f_{\theta'})^{-1}\left( \mathcal{K}_n(f_{\theta'})- \mathcal{K}_n(f_\theta)\right) \right\|_{2,op} 
\\ & \leq e^\rho \left\| f_{\theta'}- f_\theta \right\|_\infty.
\end{align*}

So that, to prove the equicontinuity, we may assume that $\theta$ is close enough to $\theta'$ to ensure that $\sup_{i = 1, \cdots, n}\left| \lambda_i\right| \leq \frac{1}{2}$.

We have
\begin{align*}
\frac{1}{m_n}\Big| \log \det (\mathcal{K}_n(f_{\theta'})) &- \log \det (\mathcal{K}_n(f_\theta)) \Big| 
\\ & =  \frac{1}{m_n} \left| \log \det \left( I_{G_n} - \mathcal{K}_n(f_{\theta_0})^{-1}\left( \mathcal{K}_n(f_{\theta'}) - \mathcal{K}_n(f_\theta)\right) \right) \right|
\\ &\leq \frac{1}{m_n} \sum_{i \in G_n} \left| \log(1+\lambda_i) \right|
\\ & \leq  \frac{1}{m_n} \sup_{i \in G_n} \left| \log(1+\lambda_i) \right|
\\ & \leq  2\log(2) \sup_{i \in G_n} \left| \lambda_i \right| 
%\text{ as soon as } \sup_{i = 1, \cdots, n}\left| \lambda_i\right| \leq \frac{1}{2} 
\\ & \leq  2\log(2) e^\rho \left\| f_{\theta'}- f_\theta \right\|_\infty.
\end{align*}

Furthermore, the sequence $(\int \log(f_\theta) \mathrm{d}\mu)_{n \in \mathbb{N}}$ is also equicontinuous since, using a Taylor formula,
$$\int \left| \log(f_{\theta'}) \mathrm{d}\mu - \int \log(f_\theta) \mathrm{d}\mu \right| \leq e^{\rho} \left\|f_{\theta'}-f_\theta \right\|_\infty.  $$

Now we tackle the equicontinuity of the sequences $$\left(X_n^T \mathcal{K}_n(f_\theta)^{-1}X_n\right)_{n \in \mathbb{N}},$$ $$\left(X_n^T \mathcal{K}_n(\frac{1}{f_\theta})X_n\right)_{n \in \mathbb{N}}$$ and $$\left(X_n^T \mathcal{Q}_n(\frac{1}{f_\theta})X_n\right)_{n \in \mathbb{N}}.$$ 

Notice first that, for any matrix $B \in M_n(\mathbb{R})$,
\begin{equation*}
 \frac{1}{m_n}\left|X_n^T B X_n \right| \leq \frac{1}{m_n} \left\|B\right\|_{2,op} \left|X_n^TX_n\right| .
\end{equation*}

It is thus sufficient to prove the equicontinuity of the sequences $$(\mathcal{K}_n(f_\theta)^{-1})_{n \in \mathbb{N}},$$ $$(\mathcal{K}_n(\frac{1}{f_\theta}))_{n \in \mathbb{N}}$$ and $$(\mathcal{Q}_n(f_\theta)^{-1})_{n \in \mathbb{N}},$$ for the norm $\left\| . \right\|_{2,op}$

Note that
\begin{align*}
\left\|\mathcal{K}_n(\frac{1}{f_{\theta'}}) - \mathcal{K}_n(\frac{1}{f_\theta})\right\|_{2,op} &\leq \left|\frac{1}{f_{\theta'}} -\frac{1}{f_\theta}  \right|_\infty
 \\ &\leq e^{2\rho}\left\| f_{\theta'}-f_\theta \right\|_\infty.
 \end{align*}

Then,
\begin{align*}
\left\|(\mathcal{K}_n(f_{\theta'}))^{-1} - (\mathcal{K}_n(f_\theta))^{-1}\right\|_{2,op} &\leq \left\|(\mathcal{K}_n(f_{\theta'}))^{-1} (\mathcal{K}_n(f_\theta))^{-1}\right\|_{2,op} \left\|(\mathcal{K}_n(f_{\theta'})) - (\mathcal{K}_n(f_\theta))\right\|_{2,op}
\\ & \leq  e^{2\rho} \left\| f_{\theta'}-f_\theta \right\|_\infty.
\end{align*}

Then, recall that, for any symmetric matrix $B \in M_n(\mathbb{R})$, we have $$\left\|B\right\|_{2,op} \leq \left\| B \right\|_{\infty,op} .$$
Recall also that $\mathcal{Q}_n(f_\theta) = B^{(n)}\odot \mathcal{K}_n(f_\theta)$. Denote 
\begin{eqnarray*}
\left\|\mathcal{Q}_n(\frac{1}{f_{\theta'}}) - \mathcal{Q}_n(\frac{1}{f_\theta})\right\|_{2,op}& \leq& \left\|\mathcal{Q}_n(\frac{1}{f_{\theta'}}) - \mathcal{Q}_n(\frac{1}{f_\theta})\right\|_{\infty,op}  
\\ & \leq & \sup_{i,j = 1, \cdots n} \left| B^{(n)}_{ij}\right| \left\|\mathcal{K}_n(\frac{1}{f_{\theta'}}) - \mathcal{K}_n(\frac{1}{f_\theta})\right\|_{\infty,op}  
\\ & \leq & (1+u_n) \left\|\frac{1}{f_{\theta'}}- \frac{1}{f_\theta} \right\|_{1,pol} (\text{see Assumption } \ref{hyp_unbiased}).
\end{eqnarray*}
% Here, we denote $\psi = \frac{1}{f}$ and $\phi = \frac{1}{g}$.

Since the map $f_\theta \mapsto \frac{1}{f_\theta}$ is continuous over $\mathcal{F}_\rho$, which is compact, we get the uniform equicontinuity of the map $f_\theta \mapsto X_n^T \mathcal{Q}_n(\frac{1}{f_\theta})X_n$ (for the norm $\left\|. \right\|_{1,pol}$).

This concludes the proof of Lemma \ref{equicont}
%  We have proved this lemma, and the equicontinuity of the sequence of functions $l_n(f) - l_n(g)$.
\end{proof}

% 
% \begin{lem}
%  $$ \sqrt{m_n}(l_n^{(u)})'(\theta_0) \rightarrow \mathcal{N}(0,\frac{1}{2}\int \frac{(f'_{\theta_0})^2}{f_{\theta_0}^2}\mathrm{d}\mu)$$
% \end{lem}
% 

\begin{proof} of Lemma \ref{norm}

% \begin{lem}\label{norm}
%  $$ \sqrt{m_n}(l_n^{(u)})'(\theta_0) \underset{n \rightarrow \infty}{\rightarrow} \mathcal{N}(0,\frac{1}{2}\int\left( \frac{f'_{\theta_0}}{f_{\theta_0}}\right)^2\mathrm{d}\mu).$$
% \end{lem}
% 
% 
% 
% \begin{lem}\label{var}
%  $$\left((l_n^{(u)})''(\breve{\theta}_n)\right)^{-1} \xrightarrow{n\rightarrow \infty} 2\left(\int\left(\frac{f'_{\theta_0}}{f_{\theta_0}}\right)^2\mathrm{d}\mu \right)^{-1}.$$
% \end{lem}
% 
% 
% \begin{lem}\label{fi_info}
%  The asymptotic Fisher information is :
% $$J(\theta_0) = \frac{1}{2}\int \left(\frac{f'_{\theta_0}}{f_{\theta_0}}\right)^2\mathrm{d}\mu.$$
% \end{lem}
% 

We aim at proving the asymptotic normality of $\sqrt{m_n}(l_n^{(u)})'(\theta_0)$. 

Using the Fourier transform, it is sufficient to prove that
$$\lim_n \mathbb{E}\left[\exp\left( i \sqrt{m_n}t\left( (l_n^{(u)})'(\theta_0)\right)\right) \right] = \exp\left(-\int \frac{1}{4}t^2\frac{(f'_{\theta_0})^2}{f_{\theta_0}^2}(t)\mathrm{d}\mu(t)  \right) $$

Recall that we have
$$(l_n^{(u)})'(\theta) = -\frac{1}{2}\int \frac{f'_\theta}{f_\theta}\mathrm{d}\mu + \frac{1}{2m_n}X_n^T\mathcal{Q}_n(\frac{f'_\theta}{f_\theta^2})X_n. $$

We can compute 
\begin{align*}
\sqrt{m_n}\mathbb{E}&\left[(l_n^{(u)})'(\theta_0)\right]  =   \sqrt{m_n}\left(-\frac{1}{2}\int \frac{f'_{\theta_0}}{f_{\theta_0}}\mathrm{d}\mu +
 \frac{1}{2m_n}\operatorname{Tr}\left( \mathcal{K}_n(f_{\theta_0})\mathcal{Q}_n(\frac{f'_\theta}{f_\theta^2})\right)\right)
\\& =   \sqrt{m_n}\left(-\frac{1}{2}\int \frac{f'_{\theta_0}}{f_{\theta_0}}\mathrm{d}\mu + 
\frac{1}{2m_n}\operatorname{Tr}\left( \mathcal{K}_n\left(f_{\theta_0}\frac{f'_{\theta_0}}{f_{\theta_0}^2}\right)\right) \right)  (\text{see Lemma } \ref{l:correction} )
\\& \leq  C v_n\sqrt{m_n} \underset{n \rightarrow \infty}{\rightarrow} 0 ~(\text{see Assumption } \ref{hyp_norm} ).
\end{align*}

% The equality can be obtained by derivation of the equality, true in both $AR_K$ and $MA_K$ cases,
% $$K_n(f_{\theta_0})Q_n(\frac{1}{f_\theta}) = K_n(\frac{f_{\theta_0}}{f_\theta}).$$

If we define 
$$Z_n= t\frac{1}{2m_n}X^T\mathcal{Q}_n(\frac{f'_\theta}{f_\theta^2})X,$$
and 
$$Z = t\frac{1}{2}\int \frac{f'_\theta}{f_\theta} \mathrm{d}\mu, $$
the last equality means that
$$\sqrt{m_n}\left(\mathbb{E}\left[Z_n\right] - Z \right) \rightarrow 0 .$$

This holds only if $f_{\theta_0}$ is a polynomial, or if all the $f_\theta, \theta \in \Theta$ are polynomials. This brings out that the 
second theorem holds for the $AR_P$ or $MA_P$ case. It also explains the term 'unbiased estimator' used for $\theta^{(u)}$.
%  the estimator of the approximated likelihood.

% then
% $$\sqrt{m_n}\left(\mathbb{E}\left[Z_n\right] - Z \right) \rightarrow 0 .$$

Then, it is sufficient to show
$$\lim_n \mathbb{E}\left[\exp\left( i \sqrt{m_n}\left(Z_n- \mathbb{E}\left[Z_n\right]    \right) \right)\right] = \exp\left(-\int \frac{1}{4}t^2\frac{(f'_{\theta_0})^2(t)}{f_{\theta_0}^2(t)}\mathrm{d}\mu(t)  \right).$$
If $\tau_k$ denotes the eigenvalues of the symmetric matrix 
$$M_n :=\frac{t}{2}\mathcal{K}_n(f_{\theta_0})^{\frac{1}{2}}\mathcal{Q}_n(\frac{f'_{\theta_0}}{f_{\theta_0}^2})\mathcal{K}_n(f_{\theta_0})^{\frac{1}{2}},$$
 then we can write
$$ Z_n = \frac{1}{m_n}\sum_{k = 1}^{m_n} \tau_k Y_k^2.$$
where $(Y_k)_{k \in G_n}$ has the standard Gaussian distribution on $\mathbb{R}^{m_n}$.

The independence of $Y_k$ leads to
$$ \log \left(  \mathbb{E}\left[\exp\left( i \sqrt{m_n}\left(Z_n- \mathbb{E}\left[Z_n\right]    \right) \right)\right]  \right) = - \sum_{k = 1}^{m_n} \left(i\frac{\tau_k}{\sqrt{m_n}}+
\frac{1}{2}\log(1-2i\frac{\tau_k}{\sqrt{m_n}})\right).$$

The $\tau_k$ are bounded, thanks to the following inequality:
\begin{eqnarray*}
\left\| M_n \right\|_{2,op}&=&\left\| \frac{t}{2}\mathcal{K}_n(f_{\theta_0})^{\frac{1}{2}}\mathcal{Q}_n(\frac{f'_{\theta_0}}{f_{\theta_0}^2})\mathcal{K}_n(f_{\theta_0})^{\frac{1}{2}} \right\|_{2,op}
\\ & \leq & \left\| \frac{t}{2}\mathcal{K}_n(f_{\theta_0})^{\frac{1}{2}}\right\|_{2,op} \left\|\mathcal{Q}_n(\frac{f'_{\theta_0}}{f_{\theta_0}^2})\right\|_{2,op} \left\|\mathcal{K}_n(f_{\theta_0})^{\frac{1}{2}} \right\|_{2,op}
\\ & \leq & \left\| \frac{t}{2}\mathcal{K}_n(f_{\theta_0})^{\frac{1}{2}}\right\|_{2,op} \left\|\mathcal{Q}_n(\frac{f'_{\theta_0}}{f_{\theta_0}^2})\right\|_{1,op} 
\left\|\mathcal{K}_n(f_{\theta_0})^{\frac{1}{2}} \right\|_{2,op}
\\ & \leq & e^\rho \alpha(f'_{\theta_0})\alpha(f_{\theta_0})^2(1 + u_n).
\end{eqnarray*}

The Taylor expansion of $\log(1-2\frac{\tau_k}{\sqrt{m_n}})$ gives
$$ \log \left(  \mathbb{E}\left[\exp\left( i \sqrt{m_n}\left(Z_n- \mathbb{E}\left[Z_n\right]    \right) \right)\right]  \right) = - \frac{1}{m_n}\sum_{k = 1}^{m_n} \tau_k^2 + R_n.$$
With $\left|R_n\right| \leq C \frac{1}{m_n\sqrt{m_n}} \sum_{k = 1}^{m_n} \left|\tau_k\right|^3$

Since the $\tau_k$ are bounded 
the assertion will be proved if we show that 
$$\frac{1}{m_n} \operatorname{Tr}(M_n^2)= \frac{1}{m_n}\sum_{k = 1}^{m_n} \tau_k^2 \xrightarrow{n\rightarrow \infty}  \int \frac{1}{4}t^2\frac{(f'_{\theta_0})^2(t)}{f_{\theta_O}^2(t)}\mathrm{d}\mu(t).$$
This last convergence is a consequence of Lemmas \ref{lem_hom} and \ref{lem_unbiased}.

This provides the asymptotic normality of
$\sqrt{m_n}(l_n^{(u)})'(\theta_0)$ and concludes the proof of Lemma \ref{norm}:
$$ \sqrt{m_n}(l_n^{(u)})'(\theta_0) \underset{n \rightarrow \infty}{\rightarrow} \mathcal{N}(0,\frac{1}{2}\int\left( \frac{f'_{\theta_0}}{f_{\theta_0}}\right)^2\mathrm{d}\mu).$$
\end{proof}

% \begin{lem}
%  $$\left((l_n^{(u)})''(\breve{\theta}_n)\right)^{-1} \rightarrow \frac{1}{2}\left(\int\frac{(f'_{\theta_0})^2}{f_{\theta_0}^2}\mathrm{d}\mu \right)^{-1}$$
% \end{lem}

\begin{proof} of Lemma \ref{var}

We aim now at proving the $P_{f_{\theta_0}}$-a.s. following convergence:
 $$\left((l_n^{(u)})''(\breve{\theta}_n)\right)^{-1} \underset{n \rightarrow \infty}{\rightarrow} \frac{1}{2}\left(\int\frac{(f'_{\theta_0})^2}{f_{\theta_0}^2}\mathrm{d}\mu \right)^{-1}$$

We have 
$$ (l_n^{(u)})''(\theta) = -\frac{1}{2m_n}\left(\int \frac{f''_\theta f_\theta- (f'_\theta)^2}{f_\theta^2}\mathrm{d}\mu +X_n^T\mathcal{Q}_n\left(\frac{2(f'_\theta)^2-f''_\theta f_\theta}{f^3_\theta} \right)X_n\right),$$ 

which leads to 
\begin{equation*}
(l_n^{(u)})''(\theta) \underset{n \rightarrow \infty}{\rightarrow} \frac{1}{2}\int \left(  \frac{f''_\theta f_\theta- (f'_\theta)^2}{f_\theta^2} +   
\frac{f_{\theta_0}\big(  2(f'_\theta)^2-f''_\theta f_\theta  \big)}{f^3_\theta}                 \right)\mathrm{d}\mu, P_{f_{\theta_0}}\text{-a.s.} 
\end{equation*}
Since the sequence $l_n^{(u)}$ is equicontinuous and $\breve{\theta}_n \underset{n \rightarrow \infty}{\rightarrow} \theta_0$, we obtain the desired convergence :
\begin{eqnarray*}
(l_n^{(u)})''(\breve{\theta}_n) \underset{n \rightarrow \infty}{\rightarrow} \frac{1}{2}\int \left(  \frac{(f'_{\theta_0})^2}{f_{\theta_0}^2}    \right)\mathrm{d}\mu, P_{f_{\theta_0}}\text{-a.s.} 
 \end{eqnarray*}

\end{proof}

% \begin{lem}
%  The asymptotic Fisher information may be written as :
% $$J(\theta_0) = \frac{1}{2}\int \frac{(f'_{\theta_0})^2}{f_{\theta_0}^2}\mathrm{d}\mu$$
% \end{lem}

\begin{proof} of Lemma \ref{fi_info}

We want to compute the asymptotic Fisher information.
As usual, it is sufficient to compute 
$$\frac{1}{m_n}\operatorname{Var}\left(L'_n(\theta_0) \right) = \lim_n \frac{1}{2m_n}\operatorname{Tr}(M_n(\theta_0)^2),$$
where $M_n(\theta) = \mathcal{K}_n(f_\theta)^{-1}\mathcal{K}_n(f'_{\theta})\mathcal{K}_n(f_\theta)^{-1}\mathcal{K}_n(f_{\theta_0})$.

This leads, together with Lemma \ref{lem_hom}, and Assumption \ref{hyp_graph_2} to
$$\frac{1}{m_n}\operatorname{Var}\left(L'_n(\theta_0) \right) \rightarrow \frac{1}{2}\int \frac{(f'_{\theta_0})^2}{f_{\theta_0}^2}\mathrm{d}\mu.$$
This ends the proof of the last lemma.
\end{proof}

% \newpage
% \input{Proof_pattern_graph_2010_10_18.tex}
% \bibliographystyle{plain}

% \bibliography{biblio110529.bib}

\def\cprime{$'$}

\end{document}